\documentclass[12pt]{article}
\usepackage[margin=1in]{geometry}  
\usepackage{graphicx}              
\usepackage{amsmath}               
\usepackage{amsfonts}              
\usepackage{amsthm}                
\usepackage{epstopdf}
\usepackage{slashbox}
\usepackage{color}

\newtheorem{thm}{Theorem}[section]
\newtheorem{lem}[thm]{Lemma}
\newtheorem{prop}[thm]{Proposition}
\newtheorem{cor}[thm]{Corollary}

\newcommand{\RR}{\mathbb{R}}      

\newcommand{\C}{\mathbb{C}}
\numberwithin{equation}{section}

\begin{document}

\nocite{*}

\title{Error Bounds for the Krylov Subspace Methods for Computations of Matrix Exponentials}

\author{Edward R. Scheinerman\thanks{Grant support listed here.} \\
Department of Applied Mathematics and Statistics \\
The Johns Hopkins University \\
Baltimore, Maryland 21218 USA}
\author
{
  Hao Wang
  \thanks
  {
    Department of Biomedical Engineering, University of Kentucky, Lexington, KY 40506, USA.
    E-mail: {\tt hao.wang@uky.edu}.
    Research supported in part by NSF under Grant DMS-1318633.
  }
  \and
  Qiang Ye
  \thanks
  {
    Department of Mathematics, University of Kentucky, Lexington, KY 40506, USA.
    E-mail: {\tt qye3@uky.edu}.
    Research supported in part by NSF under Grant DMS-1317424 and DMS-1318633.
  }
}
 \date{} 

\maketitle

\begin{abstract}
In this paper, we present new {\em a posteriori} and {\em a priori} error bounds for the Krylov subspace methods for computing $e^{-\tau A}v$ for a given $\tau>0$ and $v\in\C^n$, where $A$ is a large sparse non-Hermitian matrix. The {\em a priori} error bounds relate the convergence to $\lambda_{\min}\left(\frac{A+A^*}{2}\right)$, $\lambda_{\max}\left(\frac{A+A^*}{2}\right)$ (the smallest and the largest eigenvalue of the Hermitian part of $A$) and $|\lambda_{\max}\left(\frac{A-A^*}{2}\right)|$ (the largest eigenvalue in absolute value of the skew-Hermitian part of $A$), which  define a rectangular  region enclosing the field of values of $A$. In particular, our bounds explain an  observed superlinear convergence behavior where the error may first stagnate for certain iterations before it starts to converge. The special case that $A$ is skew-Hermitian is also considered. Numerical examples are given to demonstrate the theoretical bounds.
\end{abstract}



\section{Introduction}

The problem of computing matrix exponentials arises in many theoretical and practical problems. Numerous methods have been developed to efficiently compute $e^{-A}$ or its product with a vector $e^{-A}v$, where $A$ is an $n\times n$ complex matrix and $v\in\C^n$. We refer to the classical paper \cite{nineteen} of Moler and Van Loan for a survey of a general theory and numerical methods for matrix exponentials. For matrix exponential problems involving a large and sparse matrix $A$, it is usually the product of the exponential with a vector that is of interest. This arises, for example, in solving the initial value problem (\cite{ye14,ye24})
\begin{equation}
  \label{eqn:ode_nonhomogeneous}
  \dot{x}(t)=-Ax(t)+b(t),\;x(0)=x_0.
\end{equation}
See \cite{ye10,R13,R18} for some other applications.

A large number of matrix exponential problems concern a {\em positive definite} $A$ (i.e. $A+A^*$ is Hermitian positive definite), which defines a stable dynamical system \eqref{eqn:ode_nonhomogeneous} with a solution converging to a steady state. Another important class of problems involve a skew-Hermitian matrix $A$ (i.e. $A=iH$ with $H$ being Hermitian), for which \eqref{eqn:ode_nonhomogeneous} has a  norm-conserving solution. Such systems can be used to model a variety of physical problems where certain quantities such as energy are conserved. For example, a spectral method for solving the time-dependent Schr\"{o}dinger equation modeling N electrons leads to \eqref{eqn:ode_nonhomogeneous} with a skew-Hermitian matrix; see \cite{jGuan07,Park-Light,Schneider-Collins}. While we will study a general non-Hermitian $A$, we are particularly interested in these two important classes of problems, where stronger theoretical results can be derived.

The Krylov subspace methods are a powerful class of iterative algorithms for solving many large scale linear algebra problems. Initially introduced by Gallopoulos and Saad \cite{ye14,ye24}, they have also become a popular method for approximating
\begin{equation}
  \label{eqn:w_re}
  w(\tau):=e^{-\tau A}v,
\end{equation}
where $\tau\in\RR$ is a fixed parameter typically representing a time step. For the ease of notation, we will assume throughout that $\|v\|_2=1$. A comprehensive theory has been developed in the literature with error bounds demonstrating convergence of the Krylov subspace methods and its relation to certain properties of the matrix. For example, earlier results in \cite{ye14,ye24} relate convergence of the Krylov subspace methods to the norm of the matrix $\tau A$. More refined error bounds have later been derived, that provide sharper estimates of the errors by considering additional spectral information such as enclosing regions of the field of values of $A$ or positive definiteness of $A$; see \cite{reichel,ye9,ye10,kniz91,ye15,ye21,ye24} and the references contained therein. For a real symmetric positive definite matrix $A$, it has been shown in a recent work \cite{ye} that the speed of convergence is also determined by the condition number of $A$ as in the conjugate gradient method. For positive definite matrices that are not necessarily Hermitian, stronger convergence bounds have also been obtained in \cite{reichel,ye10,kniz91,ye15} in terms of  the field of values. However, most of these bounds are derived by assuming the field of values lying in a certain pre-defined region, which are not easy to apply or interpret. There is an inherited theoretical difficulty in quantitatively characterizing the influence on the convergence by the field of values, a two dimensional object. This issue also arises in the theory of the Krylov subspace methods for solving linear systems.

In this paper, we study the relation between the convergence of the Krylov subspace methods  and the field of values through its bounding rectangle $[a,b]\times[-c,c]$ where $a=\lambda_{\min}\left(\frac{A+A^*}{2}\right)$, $b=\lambda_{\max}\left(\frac{A+A^*}{2}\right)$ (the smallest and the largest eigenvalue of the Hermitian part of $A$) and $c=\left|\lambda_{\max}\left(\frac{A-A^*}{2}\right)\right|$ (the largest eigenvalue in absolute value of the skew-Hermitian part of $A$). With this approach, new {\em a priori} error bounds will be derived  in terms of $a$, $b$ and $c$. Simplified bounds will be presented for non-Hermitian positive definite matrices and skew-Hermitian matrices, which relate the speed of convergence to the size and the shape of the rectangular region. In particular, our bounds explain an interesting  observed  convergence behavior where the error may first stagnate for certain iterations before it starts to converge. Numerical examples will be presented to demonstrate the behavior of the new error bounds.

In developing our {\em a priori} error bounds, we also derive a new {\em a posteriori} error bound that is shown  to provide  a  sharp and computable estimate of the error. The main technique used in deriving new {\em a priori} error bounds  is the same as in the literature \cite{bebo14,ye3,reichel,kniz91,ye15} by constructing Faber polynomial approximation of the exponential function in a region containing the field of values. The novelty in this work is to use the Jacobi elliptic functions to construct a conformal mapping for the rectangular region that tightly encloses the fields of value and to show that this highly complicated mapping can be simplified to yield some simple final bounds.

The paper is organized as follows. In Section \ref{sec:preliminaries}, we first present some preliminaries about the Faber polynomial approximation and the Jacobi elliptic functions. In Section \ref{sec:aposteriori}, we present a new {\em a posteriori} error bound, which relates the convergence to the decay properties of functions of banded matrices. To study this decay behavior, we construct a conformal mapping in Section \ref{sec:conformal_mapping} and present our new {\em a priori} error bound in Section \ref{sec:apriori_nonhermitian}. In Section \ref{sec:apriori_skewhermitian}, we apply the same idea on skew-Hermitian matrices and derive simpler  {\em a priori} bounds. Numerical examples are presented in Section \ref{sec:numerical_examples} and some concluding remarks in Section \ref{sec:concluding_remarks}.


\section{Preliminaries}
\label{sec:preliminaries}

In this section, we briefly discuss some related results in complex analysis that will be needed.

\subsection{Faber polynomials}

Faber polynomials extend the theory of power series to domains more general than a disk. It starts with the Riemann mapping theorem \cite[Theorem 1.2]{markushevich} that states that every connected domain in the extended complex plane whose boundary contains more than one point can be mapped conformally onto a disk with its center at the origin. Let $\bar{\mathbb{C}}=\mathbb{C}\cup\{\infty\}$ be the extended complex plane and $D$ be a bounded,
closed continuum in the complex plane with boundary $\Gamma$ such that the complement of $D$ is simply connected in the extended plane and contains the point at $\infty$. A continuum is a non-empty, compact and connected subset of $\mathbb{C}$. Then there exists a function $w=\Phi(z)$ which maps the complement of $D$ conformally onto the exterior of a circle $|w|=\rho>0$ and satisfies the normalization conditions
\begin{equation}
  \label{eqn:normalization}
  \Phi(\infty)=\infty,\;\lim_{z\to\infty}\frac{\Phi(z)}{z}=1.
\end{equation}
Then, the function $\Phi(z)$ has a Laurent expansion at infinity of the form
\begin{equation*}
  \Phi(z)=z+\alpha_0+\frac{\alpha_{-1}}{z}+\cdots.
\end{equation*}
Moreover, given any integer $n>0$, $[\Phi(z)]^n$ has a Laurent expansion of the form
\begin{equation*}
  [\Phi(z)]^n=z^n+\alpha_{n-1}^{(n)}z^{n-1}+\cdots+\alpha_0^{(n)}+\frac{\alpha_{-1}^{(n)}}{z}+\cdots
\end{equation*}
at infinity \cite[p. 104]{markushevich}.
Then, we call the following polynomial containing non-negative powers of $z$ in the expansion
\begin{equation*}
  \Phi_n(z)=z^n+\alpha_{n-1}^{(n)}z^{n-1}+\cdots+\alpha_0^{(n)}
\end{equation*}
the Faber polynomials generated by  $D$.

The Faber polynomials can be used to approximate analytic functions on $D$, essentially through the power series approximation of a transformed function on $|w|\le\rho$. Let $\Psi$ be the inverse of $\Phi$ and let $C_R$ be the image under $\Psi$ of the circle $|w|=R>\rho$. We denote by $I(C_R)$ the bounded region enclosed by $C_R$. By \cite[Theorem 3.17]{markushevich}, every function $f(z)$ analytic on $I(C_{R})$ can be represented on $I(C_{R})$ as   a series of the Faber polynomials
\begin{equation}
  \label{eqn:faber_series}
  f(z)=\sum_{n=0}^\infty a_n\Phi_n(z)
\end{equation}
with the coefficients $a_n=\frac{1}{2\pi i}\int_{|w|=R}\frac{f[\Psi(w)]}{w^{n+1}}dw$. The partial sum of the above series
\begin{equation}
  \label{eqn:faber_sum}
  \Pi_N(z)=\sum_{n=0}^N a_n\Phi_n(z)
\end{equation}
is a polynomial of degree at most $N$ that we can use to approximate $f(z)$ on $I(C_{R})$. The next theorem of \cite{razouk54} presents some approximation bounds concerning $\Pi_N$. We first need to introduce the definition of total rotation of the boundary. For this, we assume $D$ is a closed Jordan region, i.e. its boundary $\Gamma$ is rectifiable. Then there exists a tangent vector that makes an angle $\Theta(z)$ with the positive real axis at almost all points $z\in\Gamma$. We say that $\Gamma$ has bounded total rotation $V$ if $V=\int_\Gamma|d\Theta(z)|<\infty$. Then $V\geq2\pi$ and the equality holds if $D$ is convex; see \cite{razouk54}.

\begin{thm}
  \label{thm:razouk54}
  \cite[Corollary 2.2]{razouk54}
  Assume $D$ is a closed Jordan region whose boundary $\Gamma$ has bounded total rotation $V$. For any $R>\rho$, let $f$ be an analytic function in $I(C_R)$. We have for any $N\geq0$,
  \begin{equation}
    \label{eqn:razouk54}
    ||f-\Pi_N||_\infty\leq\frac{M(R)V}{\pi}\frac{\left(\frac{\rho}{R}\right)^{N+1}}{1-\frac{\rho}{R}},
  \end{equation}
  where $M(R)=\max\limits_{z\in C_R}|f(z)|$ and $||\cdot||_\infty$ denotes the uniform norm on $I(C_R)$.
\end{thm}

Theorem \ref{thm:razouk54} is stated with $C_R$ defined from the conformal map $\Phi$ satisfying the normalization condition \eqref{eqn:normalization}. In the literature (see \cite{reichel} for example), another normalization has also been used and may be more convenient in our application. We may consider a conformal map $\widehat\Phi$ that maps the exterior of $D$ onto the exterior of the unit disk (i.e. requiring $\rho=1$ rather than \eqref{eqn:normalization}). The above theorem can be adapted to $\widehat\Phi$ through a simple normalization transformation. Namely, given $\widehat\Phi$, let $\rho=\lim\limits_{z\to\infty}\frac{z}{\widehat\Phi(z)}$ and $\Phi(z):={\rho}\widehat\Phi(z)$, where we assume $\rho$ is finite. Then $\Phi$ satisfies the normalization condition \eqref{eqn:normalization} but now maps the exterior of $D$ onto the exterior of the disk $|w|=\rho$. Applying Theorem \ref{thm:razouk54} to $\Phi$, \eqref{eqn:razouk54} holds for any $R>\rho$. Let $r:=R/\rho>1$. Let $C_R$ be the inverse image under $\Phi$ of the circle $|w|=R$ and $\widehat C_r$ be the inverse image under $\widehat\Phi$ of the circle $|w|=r$. It is easy to check that $C_R=\widehat C_r$ and then $M(R):=\max\limits_{z\in C_R}|f(z)|=\max\limits_{z\in\widehat C_r}|f(z)|$. Thus, \eqref{eqn:razouk54} is reduced to
\begin{equation}
  \label{eqn:rbound}
  ||f-\Pi_N||_\infty\leq\frac{\widehat M (r)V}{\pi}\frac{\left(\frac{1}{r}\right)^{N+1}}{1-\frac{1}{r}},
\end{equation}
where $\widehat M (r) :=\max\limits_{\widehat\Phi(z)=r}|f(z)|$. Namely, Theorem \ref{thm:razouk54} holds verbatim for a conformal map that is normalized to map the exterior of $D$ onto the exterior of the unit disk. We note however that $\rho$ as defined in the two normalizations is invariant and is called logarithmic capacity of $D$.

\subsection{Jacobi elliptic functions}

In this subsection, we introduce the Jacobi elliptic functions, which will be used to construct a conformal mapping in Section \ref{sec:apriori_nonhermitian}. More details about the Jacobi elliptic functions can be found in \cite{handbook}.

Elliptic functions were first introduced as inverse functions of (incomplete) elliptic integrals. So before the introduction of the Jacobi elliptic functions, we first state the definition and properties of elliptic integrals. Given $\phi\in\mathbb{C}$ and a real parameter $m$ with $0<m<1$, the (incomplete) Jacobi elliptic integral of the first kind is defined as
\begin{equation}
  \label{eqn:jacobi_f}
  F(\phi,m):=\int_0^\phi(1-m\sin^2\theta)^{-\frac{1}{2}}d\theta.
\end{equation}
The (incomplete) Jacobi elliptic integral of the second kind is defined as
\begin{equation*}
  E(\phi,m):=\int_0^\phi(1-m\sin^2\theta)^\frac{1}{2}d\theta.
\end{equation*}
When $\phi=\frac{\pi}{2}$, the corresponding integrals
\begin{align*}
  K(m) & :=F\left(\frac{\pi}{2},m\right)=\int_0^\frac{\pi}{2}(1-m\sin^2\theta)^{-\frac{1}{2}}d\theta, \\
  E(m) & :=E\left(\frac{\pi}{2},m\right)=\int_0^\frac{\pi}{2}(1-m\sin^2\theta)^\frac{1}{2}d\theta
\end{align*}
are called the complete Jacobi elliptic integrals of the first kind and the second kind.
Let $m_1:=1-m$, the complementary parameter of $m$. Then, $0<m_1<1$. For simplicity, we shall use the following notations.
\begin{equation}
  \label{eqn:jacobi_complete}
  \begin{aligned}
    K & :=K(m), \;\; K':=K(m_1)=K(1-m); \\
    E & :=E(m), \;\; E':=E(m_1)=E(1-m).
  \end{aligned}
\end{equation}

We now introduce the Jacobi elliptic functions. There are a total of twelve Jacobi elliptic functions in the family, but we will only discuss the basic three of them that will be used in this work. If $u=F(\phi,m)$ where $F(\phi,m)$ is the incomplete elliptic integral of the first kind defined in \eqref{eqn:jacobi_f}, three of the Jacobi elliptic functions are defined as
\begin{equation}
  \label{eqn:jacobi_functions}
  \begin{aligned}
    sn(u|m) & :=\sin\phi \\
    cn(u|m) & :=\cos\phi \\
    dn(u|m) & :=\sqrt{1-m\sin^2\phi}
  \end{aligned}
\end{equation}
The notations $sn(\sigma|m)$, $cn(\sigma|m)$ and $dn(\sigma|m)$ indicate that $sn$, $cn$ and $dn$ are functions of two independent arguments: a complex argument $u$ and a real parameter $m\in(0,1)$. Furthermore, for a fixed $m\in(0,1)$, $sn(u):=sn(u|m)$, $cn(u):=cn(u|m)$ and $dn(u):=dn(u|m)$ are doubly periodical meromorphic functions defined on $u\in\mathbb{C}$ \cite[p. 14]{tables}.

In later sections, we will need some properties of the Jacobi elliptic integrals and Jacobi elliptic functions. We summarize them in the proposition below. For details, see \cite{handbook}, \cite{dictionary} and \cite{tables}.

\begin{prop}
  \label{prop:jacobi}
  \begin{enumerate}

  \item
  $K =K(m)$ and $E=E(m)$ are positive-valued functions of $m$. Moreover, they are differentiable with respect to the parameter $m\in(0,1)$, and
  \begin{align}
    \label{eqn:derivative_k}
    \frac{dK}{dm} & =\frac{E-m_1K}{2mm_1}, \\
    \label{eqn:derivative_e}
    \frac{dE}{dm} & =\frac{E-K}{2m}.
  \end{align}

  \item
  \cite[17.3.26, p. 591]{handbook}
  \begin{equation}
    \label{eqn:17.3.26}
    \lim_{m\to1}\left[K-\frac{1}{2}\ln\left(\frac{16}{m_1}\right)\right]=0
  \end{equation}

  \item
  \cite[17.4.5, p. 592]{handbook}
  \begin{equation}
    \label{eqn:17.4.5}
    E(u+2iK')=E(u)+2i(K'-E')
  \end{equation}

  \item $sn$, $cn$ and $dn$ satisfy
  \begin{align*}
    sn^2(u|m)+cn^2(u|m) & =1 \\
    m\cdot sn^2(u|m)+dn^2(u|m) & =1
  \end{align*}

  \item
  \cite[Table 16.2, p. 570]{handbook}
  $sn$, $cn$ and $dn$ are one-valued, doubly-periodic functions. For any $l,n\in\mathbb{Z}$,
  \begin{align*}
    sn(u+2lK+2niK'|m) & =(-1)^lsn(u|m) \\
    cn(u+2lK+2niK'|m) & =(-1)^{l+n}cn(u|m) \\
    dn(u+2lK+2niK'|m) & =(-1)^ndn(u|m)
  \end{align*}

  \item
  \cite[Table 16.8, p. 572]{handbook}
  \begin{align}
    \nonumber
    sn(2iK'-\sigma|m) & =sn(-\sigma|m)=-sn(\sigma|m) \\
    \nonumber
    cn(2iK'-\sigma|m) & =-cn(-\sigma|m)=-cn(\sigma|m) \\
    \label{eqn:dn_sigma_negative}
    dn(2iK'-\sigma|m) & =-dn(-\sigma|m)=-dn(\sigma|m)
  \end{align}

  \item
  \cite[Table 16.16, p. 574]{handbook}
  Derivatives:
  \begin{align}
    \label{eqn:derivative_sn}
    \frac{d}{du}sn(u|m) & =cn(u|m)\cdot dn(u|m) \\
    \label{eqn:derivative_cn}
    \frac{d}{du}cn(u|m) & =-sn(u|m)\cdot dn(u|m) \\
    \label{eqn:derivative_dn}
    \frac{d}{du}dn(u|m) & =-m\cdot sn(u|m)\cdot cn(u|m)
  \end{align}

  \item
  \cite[16.21, p. 575]{handbook}
  Let $u=x+iy$ where $x,y\in\mathbb{R}$ and denote
  \begin{align*}
    & s=sn(x|m),c=cn(x|m),d=dn(x|m), \\
    & s_1=sn(y|m_1),c_1=cn(y|m_1),d_1=dn(y|m_1),
  \end{align*}
  Then
  \begin{align}
    \label{eqn:complex_sn}
    sn(x+iy|m) & =\frac{s\cdot d_1+ic\cdot d\cdot s_1\cdot c_1}{c_1^2+ms^2\cdot s_1^2} \\
    \label{eqn:complex_cn}
    cn(x+iy|m) & =\frac{c\cdot c_1+is\cdot d\cdot s_1\cdot d_1}{c_1^2+ms^2\cdot s_1^2} \\
    \label{eqn:complex_dn}
    dn(x+iy|m) & =\frac{d\cdot c_1\cdot d_1+ims\cdot c\cdot s_1}{c_1^2+ms^2\cdot s_1^2}
  \end{align}

  \end{enumerate}
\end{prop}

We will also need to use the signs of the real and imaginary parts of $sn(u|m)$, $cn(u|m)$ and $dn(u|m)$ when $m\in(0,1)$ and $u\in\mathbb{C}$ is in the rectangular domain $[-K,K]\times[0,2iK']$ (i.e. $\operatorname{Re}(u)\in[-K,K]$ and $\operatorname{Im}(u)\in[0,2K']$). This is discussed in \cite[pp. 172-176]{dictionary} and we summarize it in   Table \ref{tab:sn}, \ref{tab:cn} and \ref{tab:dn} for easy future references.

\begin{table}[h]
  \centering
  \begin{tabular}{|c|c|c|}
    \hline
    \backslashbox{$\operatorname{Im}(u)$}{$\operatorname{Re}(u)$} & $(-K,0)$ & $(0,K)$ \\
    \hline
    $(K',2iK')$ & $(-,-)$ & $(+,-)$ \\
    \hline
    $(0,K')$ & $(-,+)$ & $(+,+)$ \\
    \hline
  \end{tabular}
  \caption{Signs of $(\operatorname{Re}(sn(u|m)), \operatorname{Im}(sn(u|m)))$}
  \label{tab:sn}
\end{table}

\begin{table}[h]
  \centering
  \begin{tabular}{|c|c|c|}
    \hline
    \backslashbox{$\operatorname{Im}(u)$}{$\operatorname{Re}(u)$} & $(-K,0)$ & $(0,K)$ \\
    \hline
    $(K',2iK')$ & $(-,+)$ & $(-,-)$ \\
    \hline
    $(0,K')$ & $(+,+)$ & $(+,-)$ \\
    \hline
  \end{tabular}
  \caption{Signs of $(\operatorname{Re}(cn(u|m)), \operatorname{Im}(cn(u|m)) )$}
  \label{tab:cn}
\end{table}

\begin{table}[h]
  \centering
  \begin{tabular}{|c|c|c|}
    \hline
    \backslashbox{$\operatorname{Im}(u)$}{$\operatorname{Re}(u)$} & $(-K,0)$ & $(0,K)$ \\
    \hline
    $(K',2iK')$ & $(-,+)$ & $(-,-)$ \\
    \hline
    $(0,K')$ & $(+,+)$ & $(+,-)$ \\
    \hline
  \end{tabular}
  \caption{Signs of $(\operatorname{Re}(sn(u|m)), \operatorname{Im}(sn(u|m)))$}
  \label{tab:dn}
\end{table}


\section{{\em A posteriori} error bound}
\label{sec:aposteriori}

In this section, we first introduce the Arnoldi method for approximating $ w(\tau)=e^{-\tau A}v$ and then discuss an {\em a posteriori} error bound. Given $A\in\C^{n\times n}$ and $v\in\mathbb{C}^{n}$ with $||v||_2=1$, $k$ iterations of the Arnoldi process generates an orthonormal basis $\{v_1,v_2,\cdots,v_k,v_{k+1}\}$ for the Krylov subspace $K_{k+1}(A,v)=span\{v,Av,A^2v,\cdots,A^kv\}$ by
\begin{equation*}
  h_{k+1,k}v_{k+1}=Av_k-\sum_{i=1}^kh_{i,k}v_i, \;\; h_{k+1,k}\ge0.
\end{equation*}
Simultaneously, a $k$-by-$k$ upper Hessenberg matrix $H_k=[h_{ij}]$ is generated satisfying
\begin{equation}
  \label{eqn:arnoldi}
  AV_k=V_kH_k+h_{k+1,k}v_{k+1}e_k^T,
\end{equation}
where $V_k=[v_1,v_2,\cdots,v_k]$ and $e_k\in\mathbb{R}^n$ is the $k$-th coordinate vector. We note that
\begin{equation}
  \label{eqn:hkplus1}
  h_{k+1,k}^2=\|Av_k\|^2-\sum_{i=1}^kh_{i,k}^2\le\|A\|^2.
\end{equation}

We can approximate $w(\tau)=e^{-\tau A}v$ by its orthogonal projection on $K_k(A,v)$, $V_kV_k^Te^{-\tau A}v$, which is further approximated as
\begin{equation*}
  V_kV_k^Te^{-\tau A}v=V_kV_k^Te^{-\tau A}V_ke_1\approx V_ke^{-\tau V_k^TAV_k}e_1=V_k^Te^{-\tau H_k}e_1.
\end{equation*}
We call
\begin{equation}
  \label{eqn:wk_re}
  w_k(\tau):=V_k^Te^{-\tau H_k}e_1
\end{equation}
the Arnoldi approximation to $w(\tau)$ in \eqref{eqn:w_re}; see \cite{ye14,ye24}.

Let $W(A):=\{x^*Ax:x\in\C^n;\|x\|_2=1\}$ be the field of values of $A$ and $\mu(A):=\max\left\{\operatorname{Re}(z):z\in W(A)\right\}$ be the logarithmic norm of $A$ (associated with the Euclidean inner product). We also define $\nu(A):=-\mu(-A)=\min\left\{\operatorname{Re}(z):z\in W(A)\right\}$. Then we have
\begin{equation}
  \label{eqn:nuA}
  \mu(A)=\lambda_{\max}\left(\frac{A+A^*}{2}\right)\;\mbox{ and }\; \nu(A)=\lambda_{\min}\left(\frac{A+A^*}{2}\right),
\end{equation}
where $\lambda_{\max}$ and $\lambda_{\min}$ denote the largest and the smallest eigenvalues respectively. In this notation, $A$ is positive definite if and only if $\nu(A)>0$. An important property associated with the logarithmic norm \cite{lognorm,lognorm1} is that for $t\ge0$,
\begin{equation}
  \label{eq:lognorm}
  ||e^{tA}||\leq e^{t\mu(A)}.
\end{equation}

We now present a bound on the approximation error $||w(\tau)-w_k(\tau)||$ in terms of the $(k,1)$ entry of the matrix $e^{-tH_k}$.

\begin{thm}
  \label{thm:aposteriori_re}
  Let $A\in\C^{n\times n}$ and $v\in\C^n$ with $||v||=1$. Let $V_k$ be the orthogonal matrix and $H_k$ be the upper Hessenberg matrix generated by the Arnoldi process for $A$ and $v$ satisfying \eqref{eqn:arnoldi}. Let $w_k(\tau)=V_ke^{-\tau H_k}e_1$ be the Arnoldi approximation to $w(\tau)=e^{-\tau A}v$. Then the approximation error satisfies
  \begin{equation}
    \label{eqn:aposteriori_re}
    ||w(\tau)-w_k(\tau)||\leq h_{k+1,k}e^{-\min\{\nu(A),0\}\tau}\int_0^\tau|h(t)|dt,
  \end{equation}
  where
  \begin{equation}
    \label{eqn:ht_re}
    h(t):=e_k^Te^{-tH_k}e_1
  \end{equation}
  is the $(k,1)$ entry of the matrix $e^{-tH_k}$ and $\nu(A)$ is defined in \eqref{eqn:nuA}.
\end{thm}

\begin{proof}
  First, we have
  $w'(t)=-Ae^{-tA}v=-Aw(t)$ and
  \begin{align*}
    w'_k(t)
    & =-V_kH_ke^{-tH_k}e_1 \\
    & =-(AV_k- h_{k+1,k}v_{k+1}e_k^T)e^{-tH_k}e_1 \\
    & =-Aw_k(t)+ h_{k+1,k}h(t)v_{k+1}.
  \end{align*}
  Let $E_k(t):=w(t)-w_k(t)$. Then
  \begin{align*}
    E'_k(t)
    & =-Aw(t)-(-Aw_k(t)+ h_{k+1,k}h(t)v_{k+1}) \\
    & =-AE_k(t)- h_{k+1,k}h(t)v_{k+1}.
  \end{align*}
  Note that $E_k(0)=w(0)-w_k(0)=v-V_ke_1=0$. Solving the initial value problem for $E_k(t)$, we have
  \begin{equation*}
    E_k(\tau)=-h_{k+1,k}\int_0^\tau h(t)e^{(t-\tau)A}v_{k+1}dt.
  \end{equation*}
  Since $\tau-t>0$ in the integral, using \eqref{eq:lognorm}, we have
  \begin{equation*}
    ||e^{(t-\tau)A}||=||e^{(\tau-t)(-A)}||\leq e^{(\tau-t)\mu(-A)}=e^{(t-\tau)\nu(A)}.
  \end{equation*}
  Then the approximation error satisfies
  \begin{align*}
    ||E_k(\tau)||
    & \leq h_{k+1,k}\left|\left|\int_0^\tau h(t)e^{(t-\tau)A}v_{k+1}dt\right|\right| \\
    & \leq h_{k+1,k}\int_0^\tau|h(t)|\cdot||e^{(t-\tau)A}|| dt \\
    & \leq h_{k+1,k}\int_0^\tau|h(t)|\cdot e^{(t-\tau)\nu(A)}dt
  \end{align*}
  Thus, if $\nu(A)\ge0$, we have $||E_k(\tau)||\leq h_{k+1,k}\int_0^\tau|h(t)|dt $. If $\nu(A)<0$, then
  \begin{equation*}
    ||E_k(\tau)||\leq h_{k+1,k}\int_0^\tau|h(t)|e^{t\nu(A)}e^{-\tau\nu(A)}dt\leq h_{k+1,k}e^{-\tau\nu(A)}\int_0^\tau|h(t)|dt.
  \end{equation*}
  This completes the proof.
\end{proof}

$h(t)$ in the above bound is computable {\em a posteriori} for any given $t$. Being the $(k,1)$ entry of the matrix $e^{-tH_k}$, it is expected to become small as $k$ increases because of a decay property associated with functions of a banded matrix (see \cite{bebo14,bebr13,ye2,ye3}). This provides an understanding of the convergence of the error. Indeed, in \S\ref{sec:apriori_nonhermitian}, we shall extend the techniques introduced in \cite{bebo14,ye3} to derive some sharp decay bounds on $h(t)$, which will result in some new {\em a priori} bounds. Before we do that, we will need to construct some conformal mapping first in the next section.

We also remark that the {\em a posteriori} bound in the theorem contains the integral of $h(t)$ that is not directly computable. For practical error estimates, we can approximate it using a quadrature rule, say, the Simpson's rule, by computing $h(t)$ at some selected discrete points. This provides a fairly sharp {\em a posteriori} error estimates; see the numerical examples in \S\ref{sec:numerical_examples}. Note that there are several {\em a posteriori} error estimates presented in \cite{ye24} derived from approximation of a different error expression, one of which is $\tau h(\tau)$.


\section{Conformal mapping}
\label{sec:conformal_mapping}

In this section, we construct a   conformal mapping which maps the exterior of a rectangle onto the exterior of a unit disk and discuss some of its properties. Given a rectangle in $\tilde{z}$-plane whose vertices are $a\pm ic$ and $b\pm ic$ where $b>a$ and $c>0$, we map the exterior of this rectangle conformally onto $|u|>1$. This can be done in the following three steps.

\begin{itemize}

  \item
  Step 1:
  \begin{equation}
    \label{eqn:phi_re_1}
    z=\phi_1(\tilde{z})=\tilde{z}-\frac{a+b}{2}
  \end{equation}
  shifts the original rectangle to a new rectangle with vertices $\pm\alpha\pm i\beta$, where $\alpha=\frac{b-a}{2}$ and $\beta=c$.

  \item
  Step 2: $\phi_2:z\mapsto w$ is defined through an auxiliary variable $\sigma$ by
  \begin{equation}
    \label{eqn:phi_re_2}
    \left\{
      \begin{aligned}
        z & =\alpha-\frac{i}{\lambda}\{E(\sigma|m)-m_1\sigma\} \\
        w & =\frac{1-dn(\sigma|m)}{\sqrt{m}sn(\sigma|m)}
      \end{aligned}
    \right.
  \end{equation}
  where $sn(\sigma|m)$, $cn(\sigma|m)$ and $dn(\sigma|m)$ are Jacobi elliptic functions and $E(\sigma|m):=\int_0^\sigma dn^2(z|m)dz$. The parameter $m$ is determined from $\alpha, \beta$ by the equation
  \begin{equation}
    \label{eqn:m_lambda}
    \frac{E-m_1K}{\beta}=\frac{E'-mK'}{\alpha},
  \end{equation}
  here $K$, $E$, $K'$ and $E'$ are functions of $m$ or $m_1:=1-m$ defined in \eqref{eqn:jacobi_complete}. The existence and uniqueness of $m$ will be shown in Lemma \ref{lem:unique} below. It is shown in \cite[p. 178]{dictionary} that $\phi_2$ conformally maps the exterior of the rectangle $[-\alpha,\alpha]\times[-\beta,\beta]$ to the upper half plane $\{\operatorname{Im}(w)>0\}$ and that the range of $\sigma$ is in the rectangle $[-K,K]\times[0,2iK']$.

  \item
  Step 3:
  \begin{equation}
    \label{eqn:phi_re_3}
    u=\phi_3(w)=\frac{i+w}{i-w}
  \end{equation}
  maps $\{\operatorname{Im}(w)>0\}$ onto $\{|u|>1\}$.

\end{itemize}

Now let
\begin{equation}
  \label{eqn:phi_re}
  \tilde{\Phi}:=\phi_3\circ\phi_2\circ\phi_1
\end{equation}
be the composition of the above three conformal mappings defined in \eqref{eqn:phi_re_1}, \eqref{eqn:phi_re_2} and \eqref{eqn:phi_re_3}. Then $\tilde{\Phi}$ maps the exterior of the rectangle $[a,b]\times[-c,c]$ conformally onto the exterior of the unit circle.

The rest of this section will present several results concerning $\tilde{\Phi}$ that we will use in the next section, but first we give a proof of existence of a unique solution of \eqref{eqn:m_lambda} that appears not readily available in the literature.

\begin{lem}
  \label{lem:unique}
  $E(m)-(1-m)K(m)\in(0,1)$ is an increasing function and $E'(m)-mK'(m)\in(0,1)$ is an decreasing function. For any $0<\alpha,\beta<+\infty$, there exists a unique $m\in(0,1)$, as a function of $\beta/\alpha$, satisfying \eqref{eqn:m_lambda}.
\end{lem}

\begin{proof}
  Let $f(m):=E-m_1K=E(m)-(1-m)K(m)$ be a function of $m\in(0,1)$. Then $E'(m)-mK'(m) = f(1-m)$. By the definition of $K(m)$ and $E(m)$, $K(0)=\frac{\pi}{2}$, $E(0)=\frac{\pi}{2}$, and then
  \begin{equation}
    \label{eqn:limit_f_m0}
    \lim_{m\to0}f(m)=0.
  \end{equation}
  Moreover, by \eqref{eqn:17.3.26},
  \begin{equation*}
    \lim_{m\to1}m_1\left[K(m)-\frac{1}{2}\ln\left(\frac{16}{m_1}\right)\right]=0,
  \end{equation*}
  and therefore
  \begin{equation*}
    \lim_{m\to1}m_1K(m)=\lim_{m\to1}m_1\ln\left(\frac{16}{m_1}\right)=\lim_{m_1\to0}m_1\ln\left(\frac{16}{m_1}\right)=0.
  \end{equation*}
  Again by the definition of $E(m)$, $E(1)=1$. Then
  \begin{equation}
    \label{eqn:limit_f_m1}
    \lim_{m\to1}f(m)=E(1)-\lim_{m\to1}m_1K(m)=1.
  \end{equation}
  By \eqref{eqn:derivative_k} and \eqref{eqn:derivative_e}, $f(m)$ is differentiable in $(0,1)$ and
  \begin{equation*}
    \frac{d}{dm}f(m)=\frac{K(m)}{2}>0.
  \end{equation*}
  So $f$ is an increasing function of $m$ over $(0,1)$. Now consider
  \begin{equation}
    \label{eqn:definition_gm}
    g(m):=\frac{f(m)}{f(1-m)}=\frac{E(m)-(1-m)K(m)}{E(1-m)-mK(1-m)}.
  \end{equation}
  By \eqref{eqn:limit_f_m0} and \eqref{eqn:limit_f_m1}, $g(m)$ is an increasing function of $m$ over $(0,1)$ with
  \begin{equation*}
    \lim_{m\to0}g(m)=0,\;\lim_{m\to1}g(m)=+\infty.
  \end{equation*}
  Then for any $0<\alpha,\beta<+\infty$, there exists a unique $m\in(0,1)$ such that $g(m)=\frac{\beta}{\alpha}$, i.e., \eqref{eqn:m_lambda}.
\end{proof}

The parameter $m$ determined by \eqref{eqn:m_lambda} is defined by the aspect ratio $\beta/ \alpha$ (or the shape) of the rectangle $[a,b]\times[-c,c]$. For example, from the proof, $m\approx0$ if the rectangle is narrowly around the real axis, while $m\approx1$ if the rectangle is nearly a vertical line in the complex plane. When $m=1/2$, the rectangle is a square.

As in \S\ref{sec:preliminaries}, we denote by $C_r$ in the $\tilde{z}$-plane the inverse image of the circle $|u|=r$ under $\tilde{\Phi}$ for a given $r>1$. We need to determine the minimum of $\operatorname{Re}(\tilde{z})$ in $C_r$, i.e. the left most point of $C_r$. First we prove a lemma about the Jacobi elliptic functions, which is a direct result of Proposition \ref{prop:jacobi}.

\begin{lem}
  \label{thm:same_sign}
  For $u=x+iy$ where $-K<x<K$ and $0<y<2K'$,
  \begin{equation*}
    \operatorname{sgn}(\operatorname{Im}(cn(u|m)))=\operatorname{sgn}(\operatorname{Im}(dn(u|m))).
  \end{equation*}
\end{lem}

\begin{proof}
  By \eqref{eqn:complex_cn} and \eqref{eqn:complex_dn},
  \begin{align*}
    \operatorname{Im}(cn(u|m)) & =\frac{sn(x|m)dn(x|m)sn(y|m_1)dn(y|m_1)}{1-dn^2(x|m)sn^2(y|m_1)} \\
    \operatorname{Im}(dn(u|m)) & =\frac{m\cdot sn(x|m)cn(x|m)sn(y|m_1)}{1-dn^2(y|m)sn^2(y|m_1)}.
  \end{align*}
  So,
  \begin{equation}
    \label{eqn:samesign1}
    \operatorname{sgn}(\operatorname{Im}(cn(u|m)))=\operatorname{sgn}(\operatorname{Im}(dn(u|m)))\cdot\operatorname{sgn}(cn(x|m)\cdot dn(x|m)\cdot dn(y|m_1))
  \end{equation}
  Write $x=F(\phi,m)$. When $-K<x<K$, we have $\phi\in(-\frac{\pi}{2},\frac{\pi}{2})$. So,
  \begin{equation}
    \label{eqn:samesign2}
    cn(x|m)=\cos\phi>0.
  \end{equation}
  By the definition of $dn(u|m)$, for any $x,y\in\mathbb{R}$,
  \begin{equation}
    \label{eqn:samesign3}
    dn(x|m)>0,\;dn(y|m_1)>0.
  \end{equation}
  Applying \eqref{eqn:samesign2} and \eqref{eqn:samesign3} to \eqref{eqn:samesign1}, we conclude that the imaginary part of $cn(u|m)$ and that of $dn(u|m)$ have the same sign.
\end{proof}

The following lemma shows that the minimum of $\operatorname{Re}(\tilde{z})$ in $C_r$ is attained at the inverse of $u=-r$.

\begin{lem}
  \label{thm:min_re}
  Let $\tilde{\Phi}:\tilde{z}\mapsto u$ be defined in \eqref{eqn:phi_re}. Let $\tilde{\Psi}:u\mapsto\tilde{z}$ be its inverse mapping and $C_r$ be the image of $|u|=r>1$ under $\tilde{\Psi}$. Then
  \begin{equation*}
    \min\{\operatorname{Re}(\tilde{z}):\tilde{z}\in C_r\}=\tilde{\Psi}(-r).
  \end{equation*}
\end{lem}

\begin{proof}
  By \eqref{eqn:phi_re_1},
  \begin{equation}
    \label{eqn:derivative_1}
    \frac{d\tilde{z}}{dz}=1.
  \end{equation}
  Recall the definition $E(\sigma|m)=\int_0^\sigma dn^2(z|m)dz$, the identities $sn^2+cn^2\equiv1$ and $m\cdot sn^2+dn^2\equiv1$, we have from \eqref{eqn:phi_re_2} that
  \begin{equation}
    \label{eqn:derivative_2}
    \frac{dz}{d\sigma}=-\frac{i}{\lambda}\{dn^2-(1-m)\}=-\frac{i}{\lambda}\{m-m\cdot sn^2\}=-\frac{i}{\lambda}\cdot m\cdot cn^2.
  \end{equation}
  Note that By \eqref{eqn:derivative_sn} and \eqref{eqn:derivative_dn}, we have $\frac{d(dn)}{d\sigma}=-m\cdot sn\cdot cn$ and $\frac{d(sn)}{d\sigma}=cn\cdot dn$. Then by \eqref{eqn:phi_re_2},
  \begin{align}
    \nonumber
    \frac{dw}{d\sigma}
    & =\frac{-(-m\cdot sn\cdot cn)\cdot\sqrt{m}\cdot cn-(1-dn)\cdot\sqrt{m}\cdot cn\cdot dn}{m\cdot sn^2} \\
    \nonumber
    & =\frac{\sqrt{m}\cdot cn\cdot(m\cdot sn^2-dn+dn^2)}{m\cdot sn^2} \\
    \nonumber
    & =\frac{\sqrt{m}\cdot cn\cdot(1-dn)}{1-dn^2} \\
    \label{eqn:derivative_3}
    & =\frac{\sqrt{m}\cdot cn}{1+dn}
  \end{align}
  By \eqref{eqn:phi_re_3}, $w=i\frac{u-1}{u+1}$ and then
  \begin{equation}
    \label{eqn:derivative_4}
    \frac{dw}{du}=\frac{2i}{(u+1)^2}.
  \end{equation}
  Combining \eqref{eqn:derivative_1}, \eqref{eqn:derivative_2}, \eqref{eqn:derivative_3} and \eqref{eqn:derivative_4}, we have
  \begin{align}
    \nonumber
    \frac{d\tilde{z}}{du}
    & =\frac{d\tilde{z}}{dz}\cdot\frac{dz}{d\sigma}\cdot\frac{d\sigma}{dw}\cdot\frac{dw}{du} \\
    \nonumber
    & =-\frac{i}{\lambda}\cdot m\cdot cn^2\cdot\frac{1+dn}{\sqrt{m}\cdot cn}\cdot\frac{2i}{(u+1)^2} \\
    \label{eqn:derivative_5}
    & =\frac{2\sqrt{m}\cdot cn(1+dn)}{\lambda(u+1)^2}.
  \end{align}
   \eqref{eqn:phi_re_3} also implies
  \begin{equation}
  \label{eqn:dn_u-1}
    w^2=-\frac{(u-1)^2}{(u+1)^2}.
  \end{equation}
  On the other hand, by \eqref{eqn:phi_re_2},
  \begin{equation}
  \label{eqn:dn_u-2}
    w^2=\frac{(1-dn)^2}{m\cdot sn^2}=\frac{(1-dn)^2}{1-dn^2}=\frac{1-dn}{1+dn}.
  \end{equation}
  So,
  \begin{equation}
    \label{eqn:dn_u}
    dn=\frac{1-w^2}{1+w^2}=\frac{(u+1)^2+(u-1)^2}{(u+1)^2-(u-1)^2}=\frac{1}{2}\left(u+\frac{1}{u}\right)
  \end{equation}
  and hence
  \begin{equation*}
    1+dn=\frac{(u+1)^2}{2u}.
  \end{equation*}
  Substituting this into \eqref{eqn:derivative_5}, we have
  \begin{equation}
    \label{eqn:derivative_6}
    \frac{d\tilde{z}}{du}=\frac{\sqrt{m}\cdot cn}{\lambda u}.
  \end{equation}
  Now let $u$ be on the circle of radius $r$ on the complex $u$-plane. Then we can write $u=re^{i\theta}$ where $-\pi<\theta\leq\pi$. Hence
  \begin{equation}
    \label{eqn:derivative_7}
    \frac{du}{d\theta}=re^{i\theta}\cdot i=iu.
  \end{equation}
  Treating $\tilde{z}\in C_r$ as a function of $\theta$, we have from \eqref{eqn:derivative_6} and \eqref{eqn:derivative_7} that
  \begin{equation}
    \label{eqn:derivative_8}
    \frac{d\tilde{z}}{d\theta}=\frac{i\sqrt{m}}{\lambda}\cdot cn(\sigma|m).
  \end{equation}
  So
  \begin{equation*}
    \frac{d(\operatorname{Re}(\tilde{z}))}{d\theta}=\operatorname{Re}\left(\frac{d\tilde{z}}{d\theta}\right)=-\frac{\sqrt{m}}{\lambda}\operatorname{Im}(cn(\sigma|m)).
  \end{equation*}
  From \eqref{eqn:dn_u} and $u=r\cos\theta+ir\sin\theta$, we write $dn(\sigma|m)$ as a function of $\theta$,
  \begin{equation*}
    dn(\sigma|m)=\frac{1}{2}\left(r+\frac{1}{r}\right)\cos\theta+\frac{i}{2}\left(r-\frac{1}{r}\right)\sin\theta.
  \end{equation*}
  So $\operatorname{Im}(dn(\sigma|m))<0$ when $\theta\in(-\pi,0)$, and $\operatorname{Im}(dn(\sigma|m))>0$ when $\theta\in(0,\pi]$. By Lemma \ref{thm:same_sign}, the imaginary part of $cn(\sigma|m)$ always has the same sign as that of $dn(\sigma|m)$. Thus, by \eqref{eqn:derivative_8}, $\frac{d(\operatorname{Re}(\tilde{z}))}{d\theta}>0$ when $\theta\in(-\pi,0)$, and $\frac{d(\operatorname{Re}(\tilde{z}))}{d\theta}<0$ when $\theta\in(0,\pi]$. The minimum value of $\operatorname{Re}(\tilde{z})$ is attained when $\theta=\pi$, i.e., $u=-r$.
\end{proof}

Next, we find the explicit form for $\tilde{\Psi}(-r)$ in Lemma \ref{thm:min_re}.

\begin{lem}
  \label{thm:psi_r_negative}
  Let $\tilde{\Phi}:\tilde{z}\mapsto u$ be the conformal mapping from the exterior of the rectangle $[a,b]\times[-c,c]$ onto the exterior of the unit disk, as
  defined in \eqref{eqn:phi_re}, and let $\tilde{\Psi}:u\mapsto\tilde{z}$ be its inverse. Then for any $r>1$, we have
  \begin{equation}
    \label{eqn:lemma-psi}
    \tilde{\Psi}(-r)=a-\frac{1}{\lambda}\int_0^{\frac{1}{2}\left(r-\frac{1}{r}\right)}\frac{\sqrt{m+t^2}}{\sqrt{1+t^2}}dt,
  \end{equation}
  where the parameters $m$ is determined by \eqref{eqn:m_lambda} and $\lambda$ is the ratio in \eqref{eqn:m_lambda}.
\end{lem}

\begin{proof}
  Recall that $\tilde{\Phi}=\phi_3\circ\phi_2\circ\phi_1$ with $\phi_1$, $\phi_2$ and $\phi_3$ the three conformal mappings defined in \eqref{eqn:phi_re_1}, \eqref{eqn:phi_re_2} and \eqref{eqn:phi_re_3}. Let
  \begin{equation}
    \label{eqn:phi_no_tilde}
    \Phi:=\phi_3\circ\phi_2
  \end{equation}
  and $\Psi$ be its inverse. Then obviously
  \begin{equation}
    \label{eqn:psi_r_negative}
    \tilde{\Psi}(-r)=\phi_1^{-1}\circ\Psi(-r)
  \end{equation}
  The proof of this lemma consists of two parts. First, we prove that for any $r>1$,
  \begin{equation}
    \label{eqn:psi_r}
    \Psi(r)=\alpha+\frac{1}{\lambda}\int_0^{\frac{1}{2}\left(r-\frac{1}{r}\right)}\frac{\sqrt{m+t^2}}{\sqrt{1+t^2}}dt.
  \end{equation}
  By the same equation \eqref{eqn:dn_u} that was derived from \eqref{eqn:phi_re_2} and \eqref{eqn:phi_re_3}, $w$ in the map can be eliminated to define $\Phi$: $z\longleftrightarrow\sigma\longleftrightarrow u$ through the auxiliary parameter $\sigma$ as
  \begin{equation}
    \label{eqn:z_sigma_dn_u}
    \left\{
      \begin{aligned}
        z(\sigma) & =\alpha-\frac{i}{\lambda}\{E(\sigma|m)-m_1\sigma\} \\
        dn(\sigma|m) & =\frac{1}{2}\left(u+\frac{1}{u}\right)
      \end{aligned}
    \right.
  \end{equation}
  To compute $\Psi(r)$, set $u=r$ above. Then the corresponding $\sigma$ satisfies
  \begin{equation}
    \label{eqn:dn_r}
    dn(\sigma|m)=\frac{1}{2}\left(r+\frac{1}{r}\right)>1.
  \end{equation}
  By Table \ref{tab:dn}, $\sigma\in\mathbb{C}$ is on the line segment connecting $0$ and $iK'$. Let
  \begin{equation}
    \label{eqn:substitution_t_z}
    t=-i\sqrt{m}\cdot sn(s|m),
  \end{equation}
  where $s$ is on the line segment connecting $0$ and $\sigma$. By Tables \ref{tab:sn}, \ref{tab:cn} and \ref{tab:dn}, $sn(s|m)$ is purely imaginary with positive imaginary part, and $cn(s|m)$ and $dn(s|m)$ are both real and positive. Then
  \begin{align*}
    & m\cdot sn^2(s|m)=-t^2, \\
    & m\cdot cn^2(s|m)=m-m\cdot sn^2(s|m)=m+t^2\Longrightarrow\sqrt{m}\cdot cn(s|m)=\sqrt{m+t^2}, \\
    & dn^2(s|m)=1-m\cdot sn^2(s|m)=1+t^2\Longrightarrow dn(s|m)=\sqrt{1+t^2}.
  \end{align*}
  By \eqref{eqn:substitution_t_z} and \eqref{eqn:derivative_sn},
  \begin{equation*}
    dt=-i\sqrt{m}\cdot cn(s|m)\cdot dn(s|m)ds,
  \end{equation*}
  then
  \begin{equation*}
    ds=\frac{dt}{-i\sqrt{m}\cdot cn(s|m)\cdot dn(s|m)}=\frac{dt}{-i\sqrt{m+t^2}\sqrt{1+t^2}}.
  \end{equation*}
  By \eqref{eqn:dn_r},
  \begin{equation*}
    m\cdot sn^2(\sigma|m)=1-dn^2(\sigma|m)=-\frac{1}{4}\left(r-\frac{1}{r}\right)^2,
  \end{equation*}
  then
  \begin{equation*}
    \sqrt{m}\cdot sn(\sigma|m)=\frac{i}{2}\left(r-\frac{1}{r}\right).
  \end{equation*}
  Thus, as $s$ moves along the positive imaginary axis from $0$ to $\sigma$, $t$ as defined by \eqref{eqn:substitution_t_z} moves along the positive real axis from $0$ to $\frac{1}{2}\left(r-\frac{1}{r}\right)$. Then
  \begin{align*}
    \Psi(r)
    & = z(\sigma) =\alpha-\frac{i}{\lambda}\{E(\sigma|m)-m_1\sigma\} \\
    & =\alpha-\frac{i}{\lambda}\left\{\int_0^\sigma dn^2(s|m)ds-m_1\sigma\right\} \\
    & =\alpha-\frac{i}{\lambda}\int_0^\sigma m\cdot cn^2(s|m)ds \\
    & =\alpha-\frac{i}{\lambda}\int_0^{\frac{1}{2}\left(r-\frac{1}{r}\right)}(m+t^2)\frac{dt}{-i\sqrt{m+t^2}\sqrt{1+t^2}} \\
    & =\alpha+\frac{1}{\lambda}\int_0^{\frac{1}{2}\left(r-\frac{1}{r}\right)}\frac{\sqrt{m+t^2}}{\sqrt{1+t^2}}dt.
  \end{align*}
  This completes the proof of the first part \eqref{eqn:psi_r}.

  We next prove for any $r>1$,
  \begin{equation}
    \label{eqn:psi_r_odd}
    \Psi(-r)=-\Psi(r).
  \end{equation}
  Let $\sigma$ and $\tilde{\sigma}$ be the auxiliary parameters in \eqref{eqn:z_sigma_dn_u} corresponding to $r$ and $-r$ respectively. Then
  \begin{equation*}
    dn(\tilde{\sigma}|m)=\frac{1}{2}\left(-r+\frac{1}{-r}\right)=-\frac{1}{2}\left(r+\frac{1}{r}\right)=-dn(\sigma|m).
  \end{equation*}
  By \eqref{eqn:dn_sigma_negative}, $\tilde{\sigma}=2iK'-\sigma$. Thus, using \eqref{eqn:17.4.5} and \eqref{eqn:m_lambda}, we get
  \begin{align*}
    \Psi(-r)& =z(\tilde{\sigma})=
    \alpha-\frac{i}{\lambda}\{E(2iK'-\sigma|m)-m_1(2iK'-\sigma)\} \\
    \nonumber
    & =\alpha-\frac{i}{\lambda}\{2i(K'-E')-E(\sigma|m)-2m_1iK'+m_1\sigma\} \\
    \nonumber
    & =\alpha-\frac{i}{\lambda}\{-2i(E'-mK')-[E(\sigma|m)-m_1\sigma]\} \\
    \nonumber
    & =\alpha-\frac{i}{\lambda}\{-2i\cdot\lambda\alpha-[E(\sigma|m)-m_1\sigma]\} \\
    & =-\alpha+\frac{i}{\lambda}\{E(\sigma|m)-m_1\sigma\} =-z(\sigma)=-\Psi(r).
  \end{align*}
  Finally, applying $\phi_1^{-1}$ to $\Psi(-r)$ as in \eqref{eqn:psi_r_negative} and noting that $\alpha=\frac{b-a}{2}$, \eqref{eqn:lemma-psi} is proved.
\end{proof}

Finally, we show that $\tilde{\Phi}$ can be normalized according to \eqref{eqn:normalization}.
\begin{lem}
  \label{thm:normalization}
  Let $\lambda$ be the ratio in \eqref{eqn:m_lambda}. We have
\begin{align*}
  \lim_{\tilde{z}\to\infty}\frac{\tilde{\Phi}(\tilde{z})}{\tilde{z}}=2\lambda>0.
\end{align*}
\end{lem}
\begin{proof}
First, by \eqref{eqn:dn_u} and $m\cdot sn^2(\sigma|m)+dn^2(\sigma|m)=1$, we have $\sqrt{m}\cdot sn(\sigma|m)=\frac{i}{2}\left(u-\frac{1}{u}\right)$. Applying it to \eqref{eqn:derivative_6}, we have
\begin{align}
  \label{eqn:dz_du}
  \frac{d\tilde{z}}{du}=\frac{i}{2\lambda}\cdot\frac{cn(\sigma|m)}{sn(\sigma|m)}\left(1-\frac{1}{u^2}\right).
\end{align}
As $\tilde{z}\to\infty$, $\sigma\to iK'$ and $u\to\infty$ (see \cite[p. 178]{dictionary}). Since
\begin{align*}
  \lim_{\sigma\to iK'}\frac{cn(\sigma|m)}{sn(\sigma|m)}
  = \lim_{\sigma\to iK'}\frac{cn'(\sigma|m)}{sn'(\sigma|m)}
  = \lim_{\sigma\to iK'}\frac{-sn(\sigma|m)dn(\sigma|m)}{cn(\sigma|m)dn(\sigma|m)}
  = -\left(\lim_{\sigma\to iK'}\frac{cn(\sigma|m)}{sn(\sigma|m)}\right)^{-1},
\end{align*}
we have $\lim\limits_{\sigma\to iK'}\frac{cn(\sigma|m)}{sn(\sigma|m)}=-i$. Applying it to \eqref{eqn:dz_du}, $\frac{d\tilde{z}}{du}\to\frac{1}{2\lambda}$ or $\frac{du}{d\tilde{z}}\to {2\lambda}$ as $\tilde{z}\to\infty$. Then $\frac{\tilde{\Phi}(\tilde{z})}{\tilde{z}}\to2\lambda$ as $\tilde{z}\to\infty$. $\lambda>0$ follows from Lemma \ref{lem:unique}.
\end{proof}


\section{{\em A priori} error bound for non-Hermitian matrices}
\label{sec:apriori_nonhermitian}

In this section, we derive new {\em a priori} error bounds for the Arnoldi approximations of $e^{-\tau A}v$. We shall bound the error in terms of the following spectral information of $A$:
\begin{equation}
  \label{eqn:definition_abc}
  \left\{
    \begin{aligned}
      a & =\min_i\left\{\lambda_i\left(\frac{A+A^*}{2}\right)\right\}=\nu(A) \\
      b & =\max_i\left\{\lambda_i\left(\frac{A+A^*}{2}\right)\right\}=\mu(A) \\
      c & =\max_i\left\{\left|\lambda_i\left(\frac{A-A^*}{2}\right)\right|\right\}
    \end{aligned}
  \right.
\end{equation}
where $\lambda_i(M)$ ($1\leq i\leq n$) are the eigenvalues of $M$. These three numbers provide a region bounding $W(A)$, the field of values of $A$, i.e. $W(A)$ is contained in the rectangle $[a,b]\times[-c,c]$.

We shall study the convergence of the Arnoldi method through bounding $|h(t)|$ (the $(k,1)$ entry of $e^{-tH_k}$) in the {\em a posteriori} bound of \S\ref{sec:aposteriori} as in \cite{ye}. As mentioned before, analytic functions of banded matrices have a decay property, i.e. their entries decreases away from the main diagonal. Sharp decay bounds were originally derived by Benzi and Golub \cite{ye2} for Hermitian matrices; see \cite{bebr13,besi15} and the references contained therein for some further improvements. Generalizations to the non-Hermitian case, which is applicable to the Hessenberg matrix $H_k$ here, have been obtained by Benzi and Razouk \cite{ye3} and Benzi and Boito \cite{bebo14}. Specifically, for non-Hermitian matrices, the Faber polynomial approximation and the conformal mappings on a circular region containing the field of value have been introduced in \cite{bebo14,ye3} to bound the decay rate. Here we will follow the same approach of \cite{bebo14,ye3}, but we will use the conformal mapping that is constructed in \S\ref{sec:conformal_mapping} so as to utilize a more precise region $[a,b]\times[-c,c]$ that encloses the field of values. By using a smaller bounding region, a stronger approximation result and hence a stronger bound are obtained as follows.

\begin{thm}
  \label{thm:bound_ht_re}
  Let $H_k$ be a $k$-by-$k$ upper Hessenberg matrix and let $h(t)=e_k^Te^{-tH_k}e_1$ be the $(k,1)$ entry of the matrix $e^{-tH_k}$. Let $a_k=\min_i\left\{\lambda_i\left(\frac{H_k+H_k^*}{2}\right)\right\}$, $b_k=\max_i\left\{\lambda_i\left(\frac{H_k+H_k^*T}{2}\right)\right\}$ and $c_k=\max_i\left\{\left|\lambda_i\left(\frac{H_k-H_k^*}{2}\right)\right|\right\}$. Then for any $q$ with $0<q<1$,
  \begin{equation}
    \label{eqn:bound_ht_re}
    |h(t)|\leq2\,Q\,\frac{q^{k-1}}{1-q} e^{-t\tilde{z}},
  \end{equation}
  where $Q=11.08$,
  \begin{equation}
    \label{eqn:tilde_z}
    \tilde{z}=a_k-\frac{1}{\lambda}\int_0^{\frac{1}{2}\left(\frac{1}{q}-q\right)}\frac{\sqrt{m+s^2}}{\sqrt{1+s^2}}ds,
  \end{equation}
  and the parameters $m$ is determined from $a_k$, $b_k$, $c_k$ by \eqref{eqn:m_lambda} and $\lambda$ is the ratio  in \eqref{eqn:m_lambda}.
\end{thm}

\begin{proof}
  Let $\tilde{\Phi}:\tilde{z}\mapsto u$ be the conformal mapping from the exterior of the rectangle $[a_k,b_k]\times[-c_k,c_k]$ onto the exterior of the unit disk, as defined in \eqref{eqn:phi_re}. For a fixed $t\ge0$, let $f(z)=e^{-tz}$. Since $f$ is an analytic function, it can be approximated by the partial sum $\Pi_{k-2} (z)$ of the series of Faber polynomials generated by $\tilde{\Phi}$ as defined in \eqref{eqn:faber_sum}. Let $r=\frac{1}{q} > 1$ and consider $C_r$, the inverse image under $\tilde{\Phi}$ of the circle $|w|=r$. Applying Theorem \ref{thm:razouk54} or \eqref{eqn:rbound}, the approximation error in $I(C_r)$ is bounded as
  \begin{equation*}
    ||f-\Pi_{k-2}||_\infty=\max_{z\in I(C_r)}|f(z)-\Pi_{k-2}(z)|\leq2\,M(r)\,\frac{(\frac{1}{r})^{{k-1}}}{1-\frac{1}{r}},
  \end{equation*}
  where $M(r)=\max\limits_{z\in C_r}|f(z)|$ and we note that the total rotation around the rectangle is $V=2\pi$. Since $\Pi_{k-2}(z)$ is a polynomial of degree ${k-2}$, $[\Pi_{k-2}(H_k)]_{k1}=e_k^T \Pi_{k-2}(H_k)e_1=0$. Then
  \begin{align*}
    |h(t)|
    & =|[f(H_k)]_{k1}|
      =|[f(H_k)]_{k1}-[\Pi_{k-2}(H_k)]_{k1}| \\
    & \leq||f(H_k)-\Pi_{k-2}(H_k)||_2 \\
    & \leq Q\cdot\max_{z\in W(H_k)}|f(z)-\Pi_{k-2}(z)|,
  \end{align*}
  where $W(H_k)$ is the field of values of $H_k$ and the last inequality is by Crouzeix's Theorem \cite{crouzeix}. Since $W(H_k)\subseteq[a_k,b_k]\times[-c_k,c_k]\subseteq C_r$, we have
  \begin{equation*}
    |h(t)|\leq Q\max_{z\in I(C_r)}|f(z)-\Pi_{k-2}(z)|\leq2\,Q\,M(r)\,\frac{\left(\frac{1}{r}\right)^{k-1}}{1-\frac{1}{r}}.
  \end{equation*}
  Now, the theorem follows from $M(r)=\max\limits_{z\in C_r}e^{-tz} =\max\limits_{z\in C_r}e^{-t\operatorname{Re}({z})}=e^{-t\tilde{z}}$, where
  \[
  \tilde{z}=\min\{\operatorname{Re}({z}):{z}\in C_r\}=\tilde{\Psi}(r)=a_k-\frac{1}{\lambda}\int_0^{\frac{1}{2}\left(\frac{1}{q}-q\right)}\frac{\sqrt{m+s^2}}{\sqrt{1+s^2}}ds
  \]
  by Lemma \ref{thm:min_re} and Lemma \ref{thm:psi_r_negative}.
\end{proof}

We remark that $Q=11.08$ is called Crouzeix's constant and it is conjectured that it can be reduced to $2$ \cite{crouzeix}. Combining the above theorem with Theorem \ref{thm:aposteriori_re} leads to the following {\em a priori} error bound in the following theorem.

\begin{thm}
  \label{thm:apriori_re}
  Let $A\in\C^{n\times n}$ and $v\in\C^n$ with $||v||=1$ and let $w_k(\tau)=V_ke^{-\tau H_k}e_1$ be the Arnoldi approximation \eqref{eqn:wk_re} to $w(\tau)=e^{-\tau A}v$. Then for any $0<q<1$, the approximation error satisfies
  \begin{equation}
    \label{eqn:apriori_re}
    ||w(\tau)-w_k(\tau)||\leq2\,Q\,\tau\,||A||\,\frac{q^{k-1}}{1-q}\,e^{{-\tau\min\{a,0\}-\tau\tilde{z}}},
  \end{equation}
  where $Q=11.08$,
  \begin{equation}
    \label{eqn:tilde_z}
    \tilde{z}=a-\frac{1}{\lambda}\int_0^{\frac{1}{2}\left(\frac{1}{q}-q\right)}\frac{\sqrt{m+s^2}}{\sqrt{1+s^2}}ds,
  \end{equation}
  the parameters $m$ is determined by \eqref{eqn:m_lambda} from $a,b,c$ of \eqref{eqn:definition_abc} and $\lambda$ is the ratio  in \eqref{eqn:m_lambda}.
\end{thm}

\begin{proof}
  First note that $H_k=V_k^TAV_k$ for an orthogonal $V_k$. Then
  \begin{equation*}
    W(H_k)\subseteq W(A)\subseteq[a,b]\times[-c,c]
  \end{equation*}
  Now, Theorem \ref{thm:bound_ht_re} holds for $h(t)=e_k^Te^{-tH_k}e_1$, and indeed, from above and following the same proof, it holds with $a$, $b$, $c$ in place of $a_k$, $b_k$, $c_k$. Namely, $|h(t)|\leq2\,Q\,\frac{q^{k-1}}{1-q}\,e^{-t\tilde{z}}$ with $\tilde{z}$ defined as in \eqref{eqn:tilde_z} but from $a$, $b$, $c$. Now, using this bound in {\em a posteriori} error bound \eqref{eqn:aposteriori_re} in Theorem \ref{thm:aposteriori_re} and noting that $h_{k+1,k}\le\|A\|_2$ (see \eqref{eqn:hkplus1}), we have that, if $\tilde{z}\ne0$,
  \begin{align*}
    ||w(\tau)-w_k(\tau)||
    & \leq h_{k+1,k}e^{-\min\{\nu(A),0\}\tau}2\,Q\,\frac{q^{k-1}}{1-q}\,\int_0^\tau e^{-t\tilde{z}}dt \\
    & \leq2\,Q\,\|A\|_2\frac{q^{k-1}}{1-q}e^{-\min\{a,0\}\tau}\frac{1-e^{-\tau\tilde{z}}}{\tilde{z}} \\
    & =2\,Q\,\|A\|_2\frac{q^{k-1}}{1-q}e^{-\min\{a,0\}\tau}e^{-\tau\tilde{z}}\frac{e^{\tau\tilde{z}}-1}{\tilde{z}} \\
    & \leq2\,Q\,\tau\|A\|_2\frac{q^{k-1}}{1-q}e^{-\tau\min\{a,0\}-\tau\tilde{z}}
  \end{align*}
  where we have used $\frac{e^x-1}{x}\le1$ for any $x\ne0$. If $\tilde{z}=0$, the integration above gives $\tau$ and the final bound holds for this case as well. So the theorem is proved.
\end{proof}

For the rest of this section, we consider the case that $A$ is positive definite (i.e. $a>0$). In that case, the bound is simplified to
\begin{equation}
  \label{eqn:apriori_re1}
  ||w(\tau)-w_k(\tau)||\leq2\,Q\,\tau\,||A||\,\frac{q^{k-1}}{1-q}\,e^{-\tau\tilde{z}},
\end{equation}
Bounding $\tilde{z}$ of \eqref{eqn:tilde_z} using $0<m<1$, we have
\begin{equation*}
  \tilde{z}\geq a-\frac{1}{\lambda}\int_0^{\frac{1}{2}\left(\frac{1}{q}-q\right)}\frac{\sqrt{1+s^2}}{\sqrt{1+s^2}}ds=a-\frac{1}{2\lambda}\left(\frac{1}{q}-q\right).
\end{equation*}
This leads to a simple but obviously crude bound. In particular, the bound can be further simplified by setting the exponent $a-\frac{1}{2\lambda}\left(\frac{1}{q}-q\right)$ to 0, i.e. $q=\frac{1}{\sqrt{a^2\lambda^2+1}+a\lambda}$. We state these as the following corollary.

\begin{cor}
  Under the the assumptions of Theorem \ref{thm:apriori_re} and that $A$ is positive definite (i.e. $a>0$), for any $0<q<1$, the approximation error satisfies
  \begin{equation*}
    ||w(\tau)-w_k(\tau)||\leq2Q\tau||A||\frac{q^{k-1}}{1-q}e^{-\tau\left\{a-\frac{1}{2\lambda}\left(\frac{1}{q}-q\right)\right\}}.
  \end{equation*}
  In particular, for $q=\frac{1}{\sqrt{a^2\lambda^2+1}+a\lambda}$, we have
  \begin{equation}
    \label{eqn:q0}
    ||w(\tau)-w_k(\tau)||\leq2Q\tau||A||\frac{q^{k-1}}{1-q},
  \end{equation}
  i.e. the error converges at least at the rate of $\frac{1}{\sqrt{a^2\lambda^2+1}+a\lambda}$.
\end{cor}

Note that $a\lambda=2a\frac{E'(m)-mK'(m)}{b-a}=2\frac{E'(m)-mK'(m)}{b/a-1}$. Since $m$ is a function of $(b-a)/c$ (see Lemma \ref{lem:unique}) and $b/a$ is the condition number of the Hermitian part of $A$, the bound relates the convergence to this condition number and the shape of the rectangle.

More generally, we can find $q=q_0$ such that $\tilde{z}=0$. Then \eqref{eqn:q0} holds with this $q_0$ and the error converges at the rate $q_0$. We call this $q_0$ the threshold convergence rate. However, this $q_0$ may not give the best bound possible among choices of $q$. Note that $q$ influences the error bound through two opposing actions of $q^k$ and $e^{-\tau\tilde{z}}$. Namely, choosing smaller $q$ results in a faster geometrically decreasing term $q^k$, but $e^{-\tau\tilde{z}}$ may be much larger to result in an overall larger bound. So the best choice of $q$ should balance the two effects and will depend  on $k$. For example, smaller $q$ may be used for larger $k$ so that the more significant decrease in $q^k$ can offset the increase in $e^{-\tau\tilde{z}}$. This suggest a superlinear convergence behavior where, as $k$ increases, the error is bounded with a smaller  rate   $q$.

In determining $q$ to be used in the bound \eqref{eqn:apriori_re1}, we consider the minimization at each step $k$ of
\begin{equation}
  \label{eqn:eq_re}
  E(q):=\frac{q^{k-1}}{1-q}e^{-\tau\tilde{z}}.
\end{equation}
Taking derivative of $E$ with respect to $q$ and using
\begin{equation*}
  \frac{d\tilde{z}}{dq}=-\frac{1}{\lambda}\frac{\sqrt{m+\frac{1}{4}\left(\frac{1}{q}-q\right)^2}}{\sqrt{1+\frac{1}{4}\left(\frac{1}{q}-q\right)^2}}\frac{1}{2}\left(-\frac{1}{q^2}-1\right)=\frac{\sqrt{m+\frac{1}{4}\left(\frac{1}{q}-q\right)^2}}{\lambda q},
\end{equation*}
we have
\begin{align*}
  \frac{dE}{dq}
  & =\frac{(k-1)q^{k-2}(1-q)-q^{k-1}(-1)}{(1-q)^2}e^{-\tau\tilde{z}}+\frac{q^{k-1}}{1-q}e^{-\tau\tilde{z}}(-\tau)\frac{d\tilde{z}}{dq} \\
  & =e^{-\tau\tilde{z}}\frac{q^{k-3}}{(1-q)^2}\left[(k-1)q+(2-k)q^2-C(1-q)\sqrt{(1-q^2)^2+4mq^2}\right],
\end{align*}
where $C=\frac{\tau}{2\lambda}$. Thus optimal $q=q(k)$ can be found by solving
\begin{equation}
  \label{eqn:solution_q}
  (k-1)q+(2-k)q^2-C(1-q)\sqrt{(1-q^2)^2+4mq^2}=0.
\end{equation}
Note that a solution $q\in(0,1)$ exists because the function in the equation is 1 when $q=1$ and $-C<0$ when $q=0$.

Finally, we discuss a special case, i.e. $m\approx0$.

\begin{cor}
  \label{thm:mto0}
  Under the assumptions of Theorem \ref{thm:apriori_re}, and $m\approx0$, the approximation error satisfies
  \begin{equation*}
    ||w(\tau)-w_k(\tau)||\leq2\,Q\,\tau\,||A||\frac{q_0^{k-1}}{1-q_0}
  \end{equation*}
  where
  \begin{equation*}
    q_0=\frac{\sqrt{\kappa}-1}{\sqrt{\kappa}+1}+O(\sqrt{m}),
  \end{equation*}
  and $\kappa=\frac{b}{a}$.
\end{cor}

\begin{proof}
  $E'=E(1-m)$ and $K'=K(1-m)$ are both functions of $m$ and have the following expansions at $m=0$ \cite[17.3.11-12, p. 591]{handbook}
  \begin{align}
    \label{eqn:expansion_ep}
    E' & =E(m_1)=E(1-m)=1-\frac{1}{4}m\ln m+O(m) \\
    \label{eqn:expansion_kp}
    K' & =K(m_1)=K(1-m)=-\frac{1}{2}\ln m+O(1)
  \end{align}
  Then $E'-mK'$ can be expanded at $m=0$ as
  \begin{equation}
    \label{eqn:e'-mk'}
    E'-mK'=1+\frac{1}{4}m\ln m+O(m).
  \end{equation}
  Since $\alpha=\frac{b-a}{2}$,
  \begin{equation*}
    \lambda=\frac{E'-mK'}{\alpha}=\frac{2}{b-a}\left(1+\frac{1}{4}m\ln m\right)+O(m).
  \end{equation*}
  Then
  \begin{equation}
   \label{eqn:a_lambda_1_left}
    a\lambda=\frac{2}{\kappa-1}\left(1+\frac{1}{4}m\ln m\right)+O(m).
  \end{equation}
  At the same time, for $0\le s\le\frac{1}{q}-q$,
  \begin{equation*}
    \frac{\sqrt{m+s^2}}{\sqrt{1+s^2}}=\frac{s}{\sqrt{1+s^2}}+O(\sqrt{m}),
  \end{equation*}
  so
  \begin{align}
    \nonumber
    \int_0^{\frac{1}{2}\left(\frac{1}{q}-q\right)}\frac{\sqrt{m+s^2}}{\sqrt{1+s^2}}ds
    & =\int_0^{\frac{1}{2}\left(\frac{1}{q}-q\right)}\frac{s}{\sqrt{1+s^2}}ds+O(\sqrt{m}) \\
    \label{eqn:a_lambda_1_right}
    & =\frac{1}{2}\left(\frac{1}{q}+q\right)-1+O(\sqrt{m}).
  \end{align}
  Let $q=q_0$ be the unique solution of
  \begin{equation}
    \label{eqn:a_lambda_1}
    a\lambda=\int_0^{\frac{1}{2}\left(\frac{1}{q}-q\right)}\frac{\sqrt{m+s^2}}{\sqrt{1+s^2}}ds,
  \end{equation}
  where the existence of $q_0$ and the uniqueness follow from the fact that the integral on the right is a function of $q$ monotonically decreasing from $\infty$ to $0$ for $0<q<1$. Using \eqref{eqn:a_lambda_1_left} and \eqref{eqn:a_lambda_1_right}, the equation is written as
  \begin{equation*}
    \frac{2}{\kappa-1}=\frac{1}{2}\left(\frac{1}{q}+q\right)-1+O(\sqrt{m}).
  \end{equation*}
  Solving this, the solution $q_0$ with $0< q_0 <1$ is
  \begin{equation*}
    q_0=\frac{\sqrt{\kappa}-1}{\sqrt{\kappa}+1}+O(\sqrt{m}).
  \end{equation*}
  Using this $q_0$ in the bound \eqref{eqn:apriori_re}, we have $\tilde{z}=0$ and the theorem is proved.
\end{proof}

Note that $m$ is determined by $\beta/\alpha$. In particular, for $m\approx0$, $E(m)$ and $K(m)$ have the expansions
\begin{align*}
  E & =E(m)=\frac{\pi}{2}-\frac{\pi}{8}m+O(m^2) \\
  K & =K(m)=\frac{\pi}{2}+\frac{\pi}{8}m+O(m^2).
\end{align*}
We also have the expansion of $E'-mK'$ in \eqref{eqn:e'-mk'}. Then
\begin{align*}
  \frac{\beta}{\alpha}=\frac{E-m_1K}{E'-mK'}=\frac{\pi}{2}m+O(m^2), \mbox{ or } c=\frac{(b-a)\pi}{4}m+O(m^2).
\end{align*}
So the above theorem applies to the case when $c/(b-a)$ is small or $A$ is nearly Hermitian.

In an earlier paper \cite{ye}, it is shown that for a symmetric positive definite matrix $A$, the approximation error satisfies
\begin{equation*}
  ||w(\tau)-w_m(\tau)||\leq\tau||A||(\sqrt{\kappa}+1)\left(\frac{\sqrt{\kappa}-1}{\sqrt{\kappa}+1}\right)^{m-1},
\end{equation*}
where $\kappa=b/a$ is the condition number of the matrix $A$. This implies a conjugate gradient like convergence rate $q=\frac{\sqrt{\kappa}-1}{\sqrt{\kappa}+1}$ regardless of the norm of the matrix. Then Theorem \ref{thm:mto0} shows that the same conclusion holds if $A$ is nearly Hermitian.


\section{{\em A priori} error bound for skew-Hermitian matrices}
\label{sec:apriori_skewhermitian}

In this section, we consider the special case that $A$ is skew-Hermitian which, as discussed in the introduction, arises in some interesting applications. We write $A=-iH$ with $H$ being an Hermitian matrix. In this case, the Arnoldi algorithm is theoretically equivalent to the Lanczos algorithm for $H$. As we will see, the error bounds for computing
\begin{equation}
  \label{eqn:w_im}
  w(\tau):=e^{i\tau H}v.
\end{equation}
is also significantly simplified.

Applying $k$ steps of the Lanczos method to $H$ and $v_1=v$ with $\|v\|=1$ (see \cite{demmel}), we obtain an orthonormal basis $\{v_1,v_2,\cdots,v_k,v_{k+1}\}$ and a $k$-by-$k$ tridiagonal matrix $T_k$ such that
\begin{equation}
  \label{eqn:lanczos}
  H V_k=V_kT_k+\beta_{k+1}v_{k+1}e_k^T,
\end{equation}
where $V_k=[v_1,v_2,\cdots,v_k]$. This is equivalent to \eqref{eqn:arnoldi} for the Arnoldi algorithm for $A=-iH$ with $H_k=-iT_k$ and $h_{k+1,k}=\beta_{k+1}$. Then, the corresponding approximation of $w(\tau)$ is
\begin{equation}
  \label{eqn:wk_im}
  w_k(\tau):=V_ke^{i\tau T_k}e_1,
\end{equation}
which we call the Lanczos approximation. Then the same {\em a posteriori} error bound of Theorem \ref{thm:aposteriori_re} holds with $h_{k+1, k}= \beta_{k+1}$ and $h(t):=e_k^Te^{itT_k}e_1$. Namely,
\begin{equation}
  \label{eqn:aposteriori_im0}
  ||w(\tau)-w_k(\tau)||\leq\beta_{k+1}\int_0^\tau|h(t)|dt\leq\|H\|\int_0^\tau|h(t)|dt
\end{equation}
Furthermore, slightly better bounds may be obtained by shifting the matrix. Specifically, for any $\alpha\in\RR$, we can consider the shifted matrix $H-\alpha I$ and correspondingly $w(\tau,\alpha):=e^{i\tau(H-\alpha I)}v=e^{-i\tau\alpha}w(\tau)$ and $w_k(\tau,\alpha):=V_ke^{i\tau(T_k-\alpha I)}e_1=e^{-i\tau\alpha}w_k(\tau)$. Since $(H-\alpha I)V_k=V_k(T_k-\alpha I)+\beta_{k+1}v_{k+1}e_k^T$, we can apply \eqref{eqn:aposteriori_im0} to $H-\alpha I$ to get
\begin{equation*}
  \|w(\tau,\alpha)-w_k(\tau,\alpha)\|\leq\|H-\alpha I\|\int_0^\tau|h(t,a)|dt
\end{equation*}
where $h(t,\alpha):=e_k^Te^{it(T_k-\alpha I)}e_1=e^{-it\alpha}h(t)$. Thus
\begin{equation}
  \label{eqn:aposteriori_im}
  \|w(\tau)-w_k(\tau)\|=\|w(\tau,\alpha)-w_k(\tau,\alpha)\|\leq\|H-\alpha I\|\int_0^\tau|h(t)|dt.
\end{equation}
We now bound $h(t)$ as in the previous section to obtain the following {\em a priori} error bound.

\begin{thm}
  \label{thm:apriori_im}
  Let $A=-iH\in C^{n\times n}$ be a skew-Hermitian matrix and $v\in \C^n$ with $||v||=1$. Then, for any $q$ with $0<q<1$, the error of the Lanczos approximation $w_k(\tau)=V_ke^{i\tau T_k}e_1$ \eqref{eqn:wk_im} satisfies
  \begin{equation}
    \label{eqn:apriori_im}
    ||w(\tau)-w_k(\tau)||\leq\frac{4\min\{1/(1-q^2),\tau\rho/{q}\}}{1-q}q^{k}e^{\tau\rho\left(\frac{1}{q}-q\right)},
  \end{equation}
  where $\rho=(\lambda_{\max}(H)-\lambda_{\min}(H))/4$ with $\lambda_{\min}(H)$ and $\lambda_{\max}(H)$ being the smallest and the largest eigenvalues of $H$ respectively.
\end{thm}

\begin{proof}
  Let $a=\lambda_{\min}(H)$ and $b=\lambda_{\max}(H)$. We first bound $h(t):=e_k^Te^{itT_k}e_1$ as in Theorem \ref{thm:bound_ht_re} by constructing a conformal map and using the Faber polynomial approximation. Let $\Phi:= \phi_3\circ\phi_2\circ\phi_1$ where $z_1=\phi_1(z)=-iz$ maps the exterior of $E:=\{i\lambda:\lambda\in[a,b]\}$ to the exterior of $[a,b]$, $z_2=\phi_2(z_1)=\frac{2}{b-a}\left(z_1-\frac{a+b}{2}\right)$ maps the exterior of $[a,b]$ to the exterior of $[-1,1]$, $w=\phi_3(z_2)=i(z_2+\sqrt{z_2^2-1})$ maps the exterior of $[-1,1]$ to $\{|w|>1\}$. In the definition of $\phi_3$, we choose the branch of $\sqrt{z^2-1}$ such that $\lim\limits_{z\mapsto\infty}\frac{\sqrt{z^2-1}}{z}=1$. Then $\Phi$ maps the exterior of $E$ to the exterior of the unit circle $\{|w|=1\}$ with $\rho:= \lim_{z\to\infty}\frac{z}{\Phi(z)}=\frac{b-a}{4}$. Construct the Faber polynomials from this conformal map $\Phi$ and the Faber polynomial approximation $\Pi_{k-2}$ of $f(z):=e^{tz}$ as defined in \eqref{eqn:faber_sum}. Let $r:=\frac{1}{q}>1$ and let $C_r$ be the inverse image under $\Phi$ of the circle $|w|=r$. Applying Theorem \ref{thm:razouk54} or \eqref{eqn:rbound}, the approximation error in $I(C_r)$ is bounded as,
  \begin{equation*}
    ||f-\Pi_{k-2}||_\infty\leq2\,M(r)\,\frac{(\frac{1}{r})^{{k-1}}}{1-\frac{1}{r}}=2 M(r)\,\frac{q^{{k-1}}}{1-q},
  \end{equation*}
  where $M(r)=\max\limits_{z\in C_r}|f(z)|$ and we note that the total rotation of $E$ (a line segment) is $V=2\pi$.

  To find $M(r)$, for any $z\in C_r$, we write $z=\Phi^{-1}(w)$ with $w=re^{i\theta}$ where $\theta\in[0,2\pi)$. Then, it follows from the definition of $\Phi$ that
  \begin{align*}
    z_2 & =\frac{1}{2}\left(-iw+\frac{1}{-iw}\right)=\frac{1}{2}\left(-i\frac{e^{i\theta}}{q}+\frac{iq}{e^{i\theta}}\right)=-\frac{i}{2}\left[\left(\frac{1}{q}-q\right)\cos\theta+i\left(\frac{1}{q}+q\right)\sin\theta\right], \\
    z_1 & =\frac{b-a}{2}z_2+\frac{b+a}{2}=\left[\frac{b-a}{4}\left(\frac{1}{q}+q\right)\sin\theta+\frac{b+a}{2}\right]-i\left[\frac{b-a}{4}\left(\frac{1}{q}-q\right)\cos\theta\right], \\
    z & =iz_1=\frac{b-a}{4}\left(\frac{1}{q}-q\right)\cos\theta+i\left[\frac{b-a}{4}\left(\frac{1}{q}+q\right)\sin\theta+\frac{b+a}{2}\right].
  \end{align*}
  Thus
  \begin{align*}
    M(r)=\max_{z\in C_r}|e^{tz}|=\max_{z\in C_r}e^{t\operatorname{Re}(z)}=e^{\frac{t(b-a)}{4}\left(\frac{1}{q}-q\right)}.
  \end{align*}

  Now, let $\lambda_j$ ($1\le j\le n$) be the eigenvalues of $iT_k$. Then $\lambda_j\subset E$. As in the proof of Theorem \ref{thm:bound_ht_re}, we have
  \begin{align*}
    |h(t)|
    & =|[f(iT_k)]_{k1}|
      =|[f(iT_k)]_{k1}-[\Pi_{k-2}(iT_k)]_{k1}| \\
    & \leq||f(iT_k)-\Pi_{k-2}(iT_k)||_2
      =\max_j |f(\lambda_j)-\Pi_{k-2}(\lambda_j)| \\
    & \leq\max_{z\in E}|f(z)-\Pi_{k-2}(z)| \leq ||f-\Pi_{k-2}||_\infty \\
    & \leq\frac{2q^{k-1}}{1-q}e^{\frac{t(b-a)}{4}\left(\frac{1}{q}-q\right)}.
  \end{align*}

   Finally, using \eqref{eqn:aposteriori_im} with $\alpha=(a+b)/2$, we have $\|H-\alpha I\|=(b-a)/2$ and hence
  \begin{align*}
    ||w(\tau)-w_k(\tau)||
    & \leq\frac{b-a}{2}\int_0^\tau\frac{2q^{k-1}}{1-q}e^{\frac{t(b-a)}{4}\left(\frac{1}{q}-q\right)}dt \\
    & =\frac{4 q^{k-1}}{(1-q)\left(\frac{1}{q}-q\right)}\left(e^{\frac{\tau(b-a)}{4}\left(\frac{1}{q}-q\right)}-1\right) \\
    & \le\frac{4q^{k}}{(1-q)\left(1-q^2\right)}\min\{1,\;\frac{\tau(b-a)}{4}\left(\frac{1}{q}-q\right)\}e^{\frac{\tau(b-a)}{4}\left(\frac{1}{q}-q\right)} \\
    & =\frac{4 q^{k}}{1-q}\min\{\frac{1}{1-q^2},\;\frac{\tau\rho}{q}\}e^{\tau\rho\left(\frac{1}{q}-q\right)}
  \end{align*}
  where we have used $e^x-1\le\min\{1,\;x\}e^x$ for any $x\ge0$.
\end{proof}

As before, we have an error bound for any given $q\in(0,1)$. Using smaller $q$ results in a faster geometrically decreasing term $q^k$, but $e^{\tau\rho\left(\frac{1}{q}-q\right)}$ is expected to be larger. So, again, we study the value of $q$ that minimizes the bound
\begin{equation}
  \label{eqn:eq_im}
  E(q):=\frac{q^k}{(1-q)(1-q^2)}e^{\tau\rho\left(\frac{1}{q}-q\right)},
\end{equation}
Taking derivative of $E(q)$ with respect to $q$ to get
\begin{equation*}
  \frac{dE}{dq}=\frac{q^{k-2}e^{\tau\rho\left(\frac{1}{q}-q\right)}}{(1-q)^3(1+q)^2}\left[\tau\rho q^4+(3-k)q^3+q^2+kq-\tau\rho\right].
\end{equation*}
With $E(q)\rightarrow\infty$ as $q\rightarrow0$ or $1$, the optimal value $q_0=q_0(k)$ that minimizes $E(q)$ is given by the solution of the equation
\begin{equation*}
  \tau\rho q^4+(3-k)q^3+q^2+kq-\tau\rho=0.
\end{equation*}
Note that it can be shown that the above equation has a unique solution $q_0\in(0,1)$ (see \cite{hwang} for details).

Note that $\frac{1}{1-q}$ in $E(q)$ is a well bounded term unless $q\approx1$. For example, it is bounded by $10$ if $q\le0.9$. To quantitatively interpret the bound, we can consider minimization of
\begin{equation}
  \label{eqn:esq_im}
  E_s(q)=q^ke^{\tau\rho\left(\frac{1}{q}-q\right)},
\end{equation}
which is essentially the same as $E(q)$ unless $q \approx 1$. Differentiate $E_s$ to get
\begin{equation*}
  \frac{dE_s}{dq}=e^{\tau\rho\left(\frac{1}{q}-q\right)}q^{k-2}\left[-\tau\rho q^2+kq-\tau\rho\right].
\end{equation*}
The discriminant of the quadratic $-\tau\rho q^2+kq-\tau\rho$ is $\Delta=k^2-4(\tau\rho)^2$. So, if $k\le2\tau\rho$, $E_s(q)$ is monotonically decreasing with the minimum occurring at $q_0 =1$. If $k>2\tau\rho$, $E_s(q)$ is minimized at $q_0 =\frac{k-\sqrt{k^2-4(\tau\rho)^2}}{2\tau\rho}<1$. Thus, the bound implies different convergence behavior at two stages of the Lanczos iterations.
\begin{enumerate}
  \item
  When $1\le k\le2\tau\rho$, there is essentially no decrease in the error bound.
  \item
  For $k>2\tau\rho$, the error bounds for subsequent steps decrease at least at the rate of $q_0$.
\end{enumerate}
The convergence behavior as implied from this theory is indeed what has been observed in the numerical examples (see \S\ref{sec:numerical_examples}), where the error initially stagnates for approximately $2 \tau \rho $ steps and then begins to decrease  superlinearly. Thus our bound qualitatively explains this convergence property observed numerically.

Finally, we note that the convergence bound for skew-Hermitian matrices have also been studied by Hochbruch and Lubich \cite[Theorem 4]{ye15}. It is proved there that for $k\geq2\rho\tau$,
\begin{equation}
  \label{eqn:hochbruck_im}
  ||w(\tau)-w_k(\tau)||\leq12e^{\frac{-(\rho\tau)^2}{k}}\left(\frac{e\rho\tau}{k}\right)^k.
\end{equation}
Interestingly, the range of validity of the bound coincides with the point of initial convergence as implied by our bound. It turns out that this bound can be implied from a special case of our error bound \eqref{eqn:apriori_im}. For $k\geq2\rho\tau$, let $q=\frac{\tau\rho}{k}\leq\frac{1}{2}$. Then our bound \eqref{eqn:apriori_im}, simply using $1/(1-q^2)$ for the minimum, reduces to \eqref{eqn:hochbruck_im} as follows:
\begin{align*}
  ||w(\tau)-w_k(\tau)||
  & \leq\frac{4\left(\frac{\tau\rho}{k}\right)^k}{(1-\frac{1}{2})(1-\frac{1}{2})^2}e^{\tau\rho\left(\frac{k}{\tau\rho}-\frac{\tau \rho}{k}\right)} \\
  & =\frac{32}{3}e^{-\frac{(\tau\rho)^2}{k}}\left(\frac{e\tau\rho}{k}\right)^k
  \leq12e^{-\frac{(\tau \rho)^2}{k}}\left(\frac{e\tau\rho}{k}\right)^k.
\end{align*}


\section{Numerical examples}
\label{sec:numerical_examples}

In this section, we present several numerical examples to demonstrate the error bounds obtained in this paper. All tests were carried out on a PC in MATLAB (R2013b) with the machine precision $\approx2e-16$. The Jacobi elliptic integrals that are needed for our bounds are computed using MATLAB built-in functions {\tt ellipticK} and {\tt ellipticE}.

We will construct several testing matrices with different spectral distributions and compare the actual approximation error with the new {\em a posterior} error estimate \eqref{eqn:aposteriori_re} and {\em a priori} bounds \eqref{eqn:apriori_re} or \eqref{eqn:apriori_im}. The integral in the {\em a posterior} error estimate \eqref{eqn:aposteriori_re} is approximated using Simpson's rule  with $10$ subintervals on $[0,\tau]$.

We shall compare our bounds with the bounds by Saad \cite{ye24} and where applicable with those of Hochbruck and Lubich \cite{ye15} as well. For example, if the matrices are positive semidefinite, we consider the following bound of Saad \cite[Cor. 2.2]{ye24}:
\begin{equation}
  \label{eqn:saad_re}
  ||w(\tau)-w_k(\tau)||\leq\frac{2}{k!}(\tau||A||)^k.
\end{equation}
and the following bound of Hochbruck and Lubich \cite[Theorem 2]{ye15}:
\begin{equation}
  \label{eqn:hochbruck_re}
  ||e^{-\tau A}v-V_ke^{-\tau H_k}e_1||\leq12e^{-\rho\tau}\left(\frac{e\rho\tau}{k}\right)^k,
\end{equation}
which holds for $k\geq2\rho\tau$ and with the assumption that the field of values $W(A)$ is contained in the disk $|z-\rho|<\rho$.

\emph{Example 1.}
Given an odd integer $N$ and a rectangle $[a,b]\times[-c,c]$ in the complex plane where $a$, $b$ and $c$ are all positive real numbers, let $A$ be the $N^2\times N^2$ block diagonal matrix with the diagonal blocks being $2\times2$ matrices $B_{\ell,j}$ for $\ell=1,2,\cdots,N$ and $j=1,2,\cdots,\frac{N-1}{2}$, where
\begin{equation*}
  B_{\ell,j} =
  \left[
    \begin{array}{cc}
      x_\ell & y_j \\
      -y_j & x_\ell
    \end{array}
  \right], \;
  x_\ell= a+\frac{(\ell-1)(b-a)}{N-1}
  \;\;\mbox{ and }\;\;
  y_j = \frac{2jc}{N-1}.
\end{equation*}
Then, the eigenvalues of $A$ are $x_l\pm i y_j$ with $i$ being the imaginary unit, which are the grid points of the $N\times N$ lattice on $[a,b]\times[-c,c]$. Clearly, $A$ is a normal matrix, so the field of values of $A$ is the convex hull of its eigenvalues, i.e., the rectangle $[a,b]\times[-c,c]$.

The primary purpose of this numerical test is to compare our {\em a priori} bound with Hochbruch and Lubich's bound \eqref{eqn:hochbruck_re}. The latter is applicable when $W(A)$ is contained in a disk $|z-\rho|<\rho$. We therefore choose $[a,b]\times[-c,c]$ to be the square $[1-\frac{\sqrt{2}}{2},1+\frac{\sqrt{2}}{2}]\times[-\frac{\sqrt{2}}{2},\frac{\sqrt{2}}{2}]$ which is enclosed in the circle $|z-1|<1$ and construct a matrix $A$ as above such that the eigenvalues of $A$ form a $31\times31$ lattice in the square. We apply the Arnoldi method to compute $e^{-\tau A}v$ where $v$ is a random normalized vector and we use $\tau=10,20,30,40$. In Figure \ref{fig:ex1}, we plot against the iteration number the actual error $||w(\tau)-w_k(\tau)||$ in the solid line, the {\em a posteriori} error estimate \eqref{eqn:aposteriori_re} in the $+$-line, our {\em a priori} bound \eqref{eqn:apriori_re} in the dashed line, Hochbruch and Lubich's bound \eqref{eqn:hochbruck_re} in the dotted line, and Saad's bound \eqref{eqn:saad_re} in the x-line. Note that Hochbruch and Lubich's bound is only valid for $k\geq2\rho\tau$.

\begin{figure}[!h]
  \centering
  \includegraphics[scale=0.4]{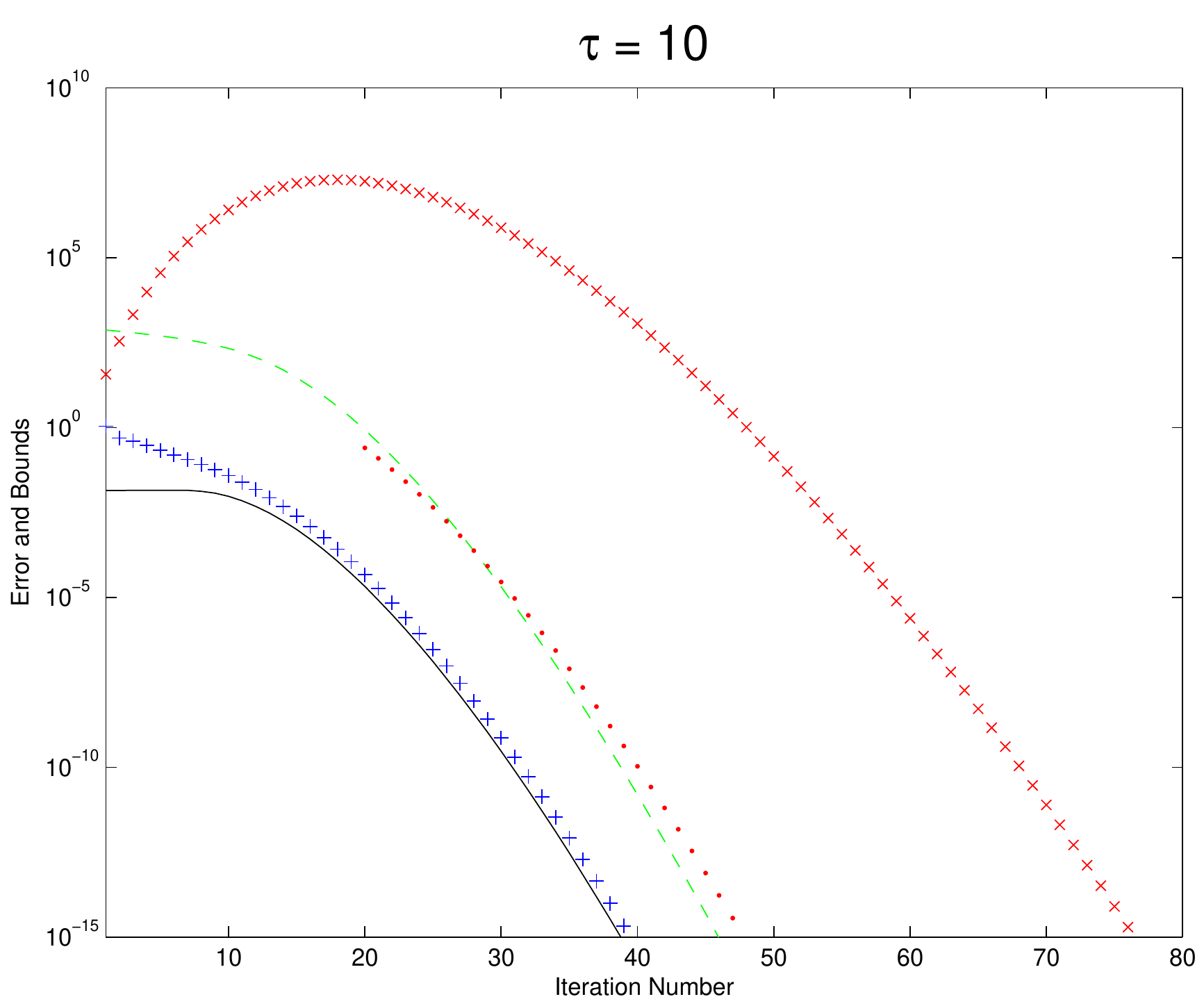}
  \includegraphics[scale=0.4]{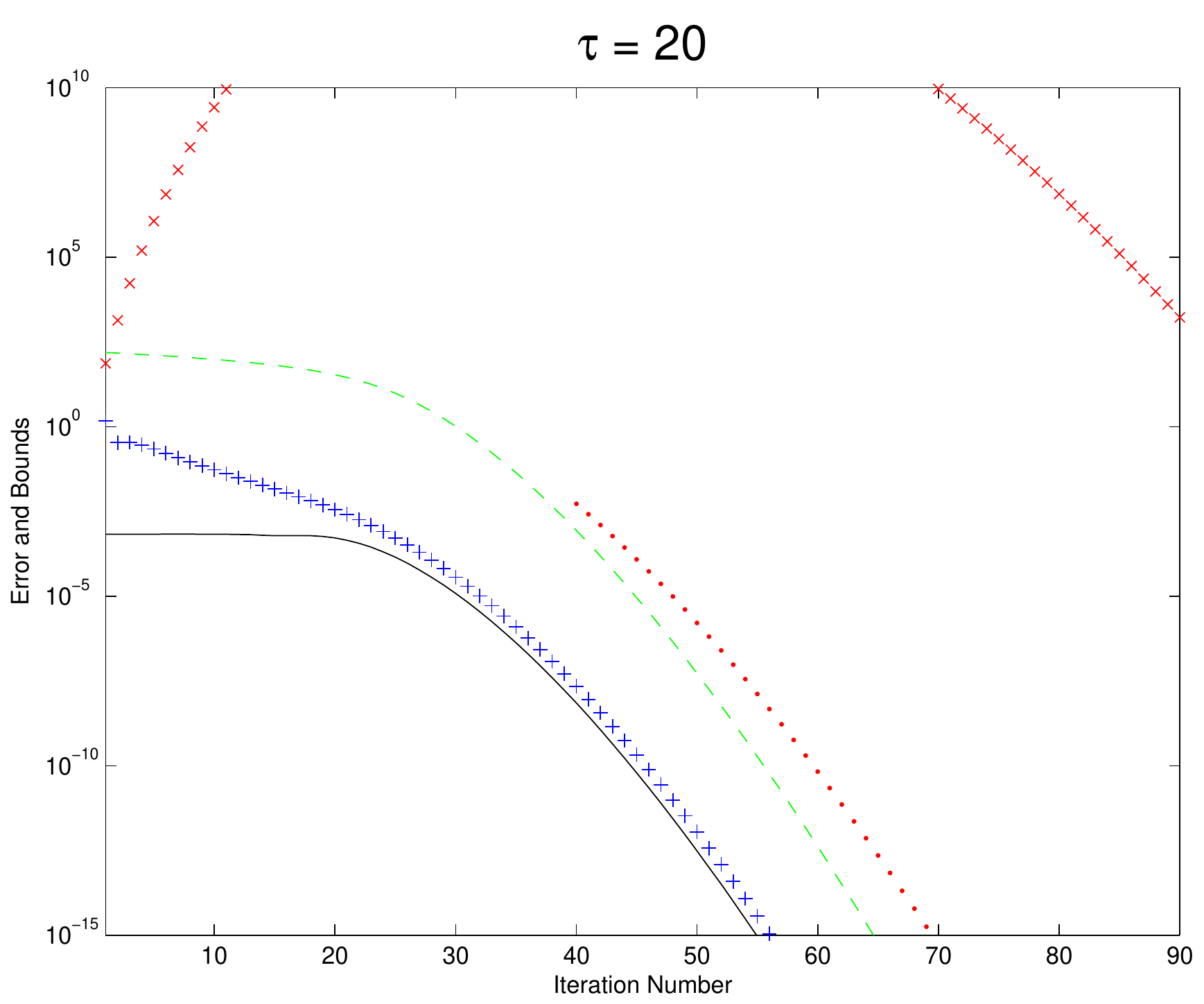}
  \includegraphics[scale=0.4]{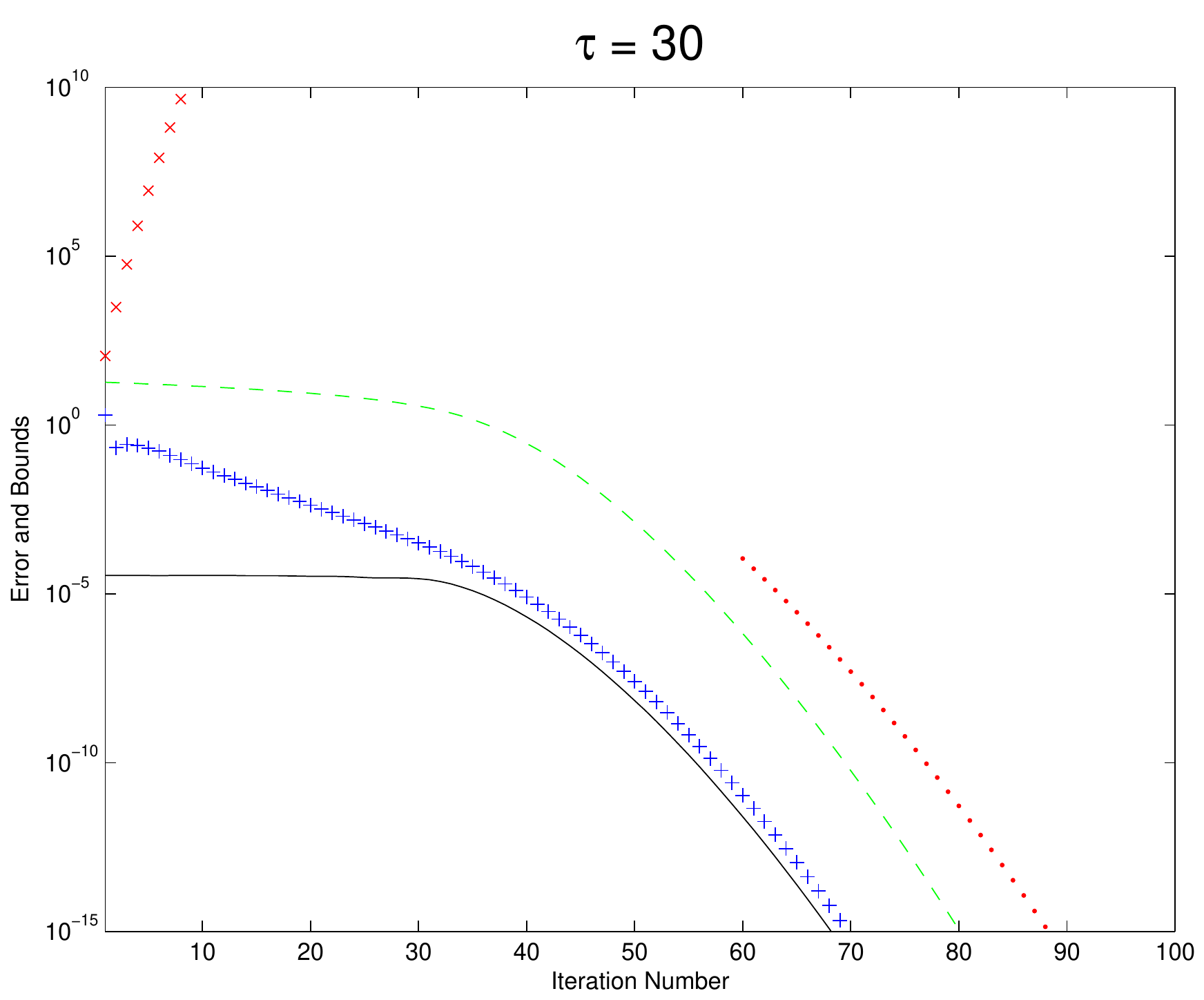}
  \includegraphics[scale=0.4]{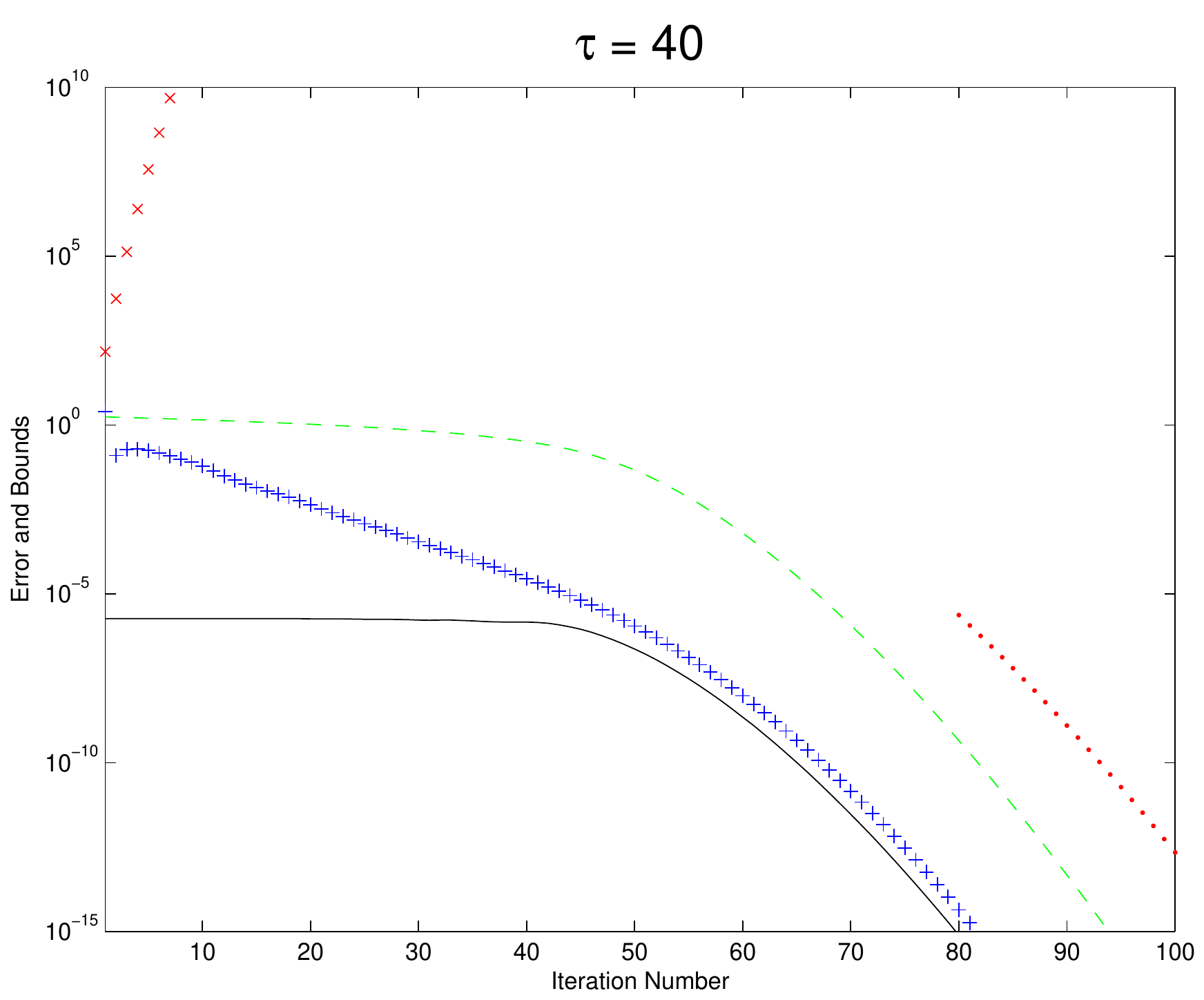}
  \caption
  {
    Example 1.
    $W(A)$ in $|z-1|<1$ and
    $\tau=10,20,30,40$.
    Error (solid), our {\em a posteriori} bound (+), our {\em a priori} bound (dashed), Saad's bound (x), and Hochbruck and Lubich's bound (dotted).
  }
  \label{fig:ex1}
\end{figure}

We observe that when $\tau$ is relatively small, our new {\em a priori} bound is comparable to Hochbruch and Lubich's bound, but as $\tau$ increases, our bound improves significantly. In particular, for larger $\tau$ values, the error $||w(\tau)-w_k(\tau)||$ first stagnates for certain number of iterations before it starts to converge. Our {\em a priori} bound nicely captures this behavior and the point where the convergence begins, while Hochbruch and Lubich's bound is pessimistic and is applicable to iterations long after the initial point of convergence. Our {\em a posteriori} error estimate is sharp at the convergence stage for all tests.

In the next example, we use the same construction as in Example 1, but consider  the field of values contained in rectangles of different shape. This is to investigate the influence on the convergence rate by the shape of the rectangle through the parameter $m$ in \eqref{eqn:m_lambda}.

\emph{Example 2.}
For a given parameter $m\in(0,1)$, we determine the dimensions of the rectangle $\alpha$ and $\beta$ by $\alpha=E'-mK',\;\beta=E-m_1K$. We then construct a matrix as in Example 1 whose field of values is contained in the rectangle $[0,2\alpha]\times[-\beta,\beta]$. We use $m\in\{0.01,0.1,0.9,0.99\}$ whose corresponding values of $\alpha, \beta$ are listed in Figure \ref{fig:ex2}. Note from Section 3.3 that $m\approx0$ means that the matrix is close to being Hermitian, and that $m\approx1$ means the matrix is close to being skew-Hermitian with a real spectral shift. We apply the Arnoldi method to compute $e^{-\tau A}v$ for a random normalized vector $v$ and we use $\tau=30$ to give $\tau A$ a moderate norm. In Figure \ref{fig:ex2} we plot the error $||w(\tau)-w_k(\tau)||$ in the solid line, our {\em a posteriori} error estimate \eqref{eqn:aposteriori_re} in $+$-line, our {\em a priori} bound \eqref{eqn:apriori_re} in the dashed line and Saad's bound \eqref{eqn:saad_re} in the x-line.

\begin{figure}[!h]
  \centering
  \includegraphics[scale=0.4]{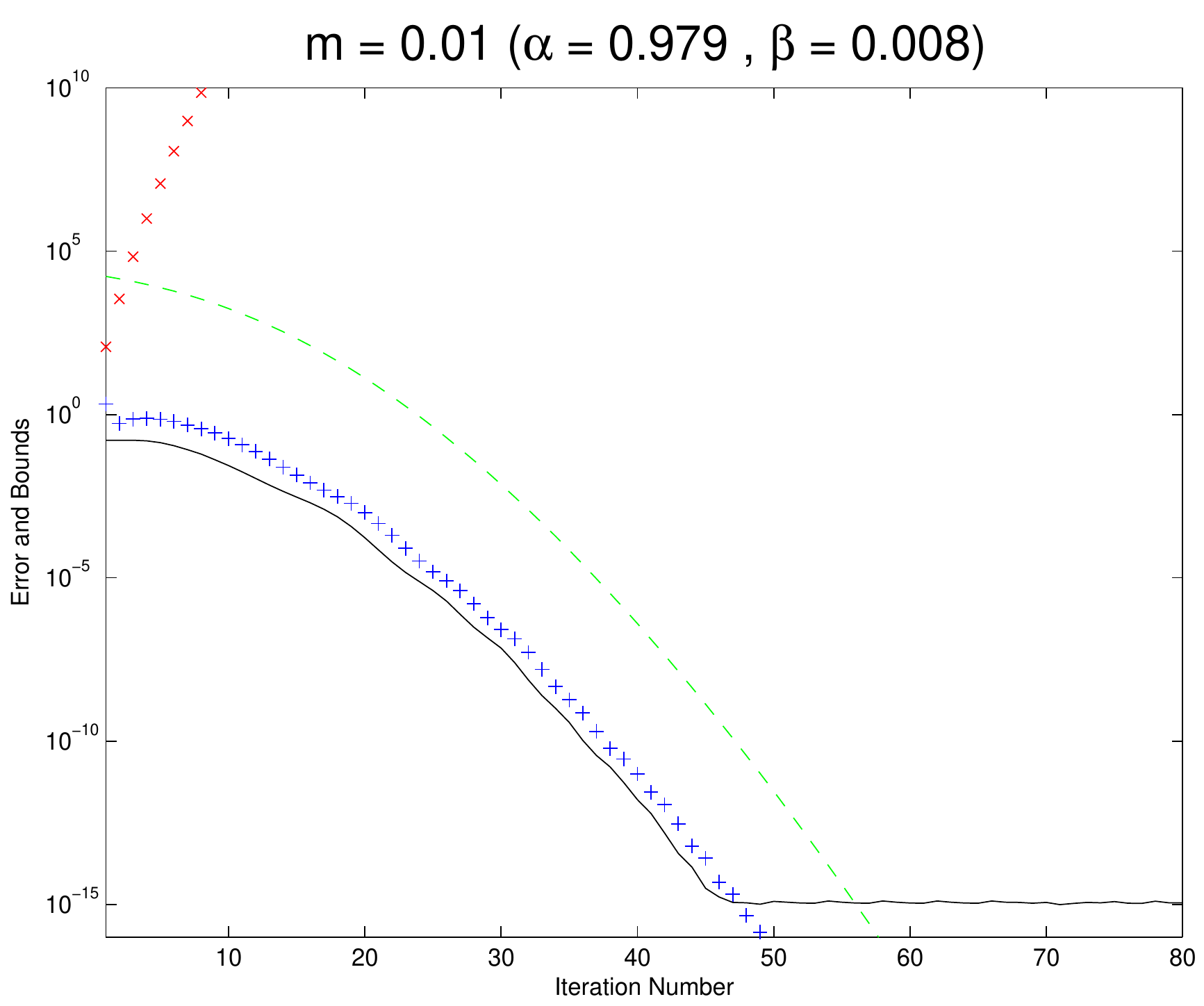}
  \includegraphics[scale=0.4]{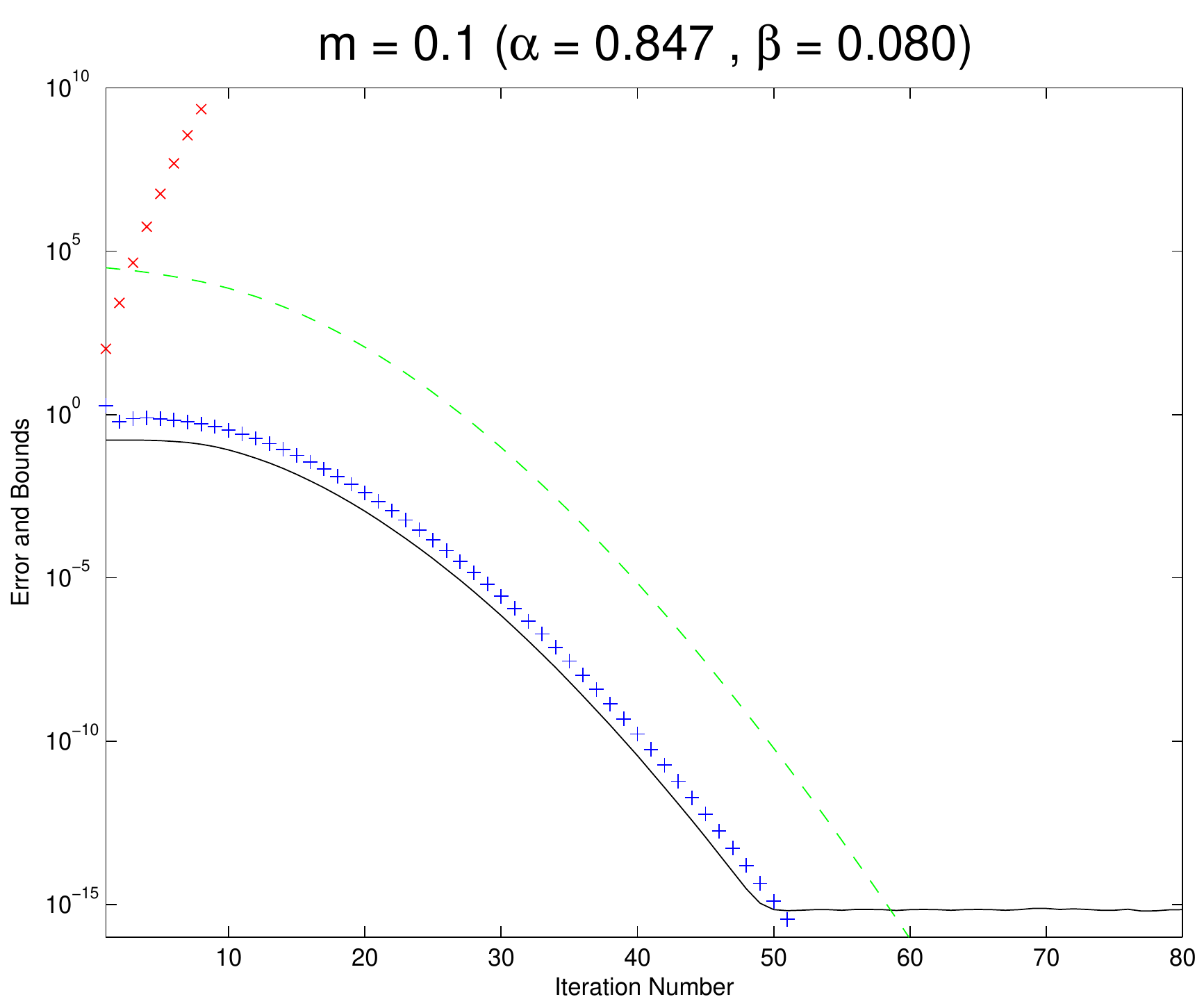}
  \includegraphics[scale=0.4]{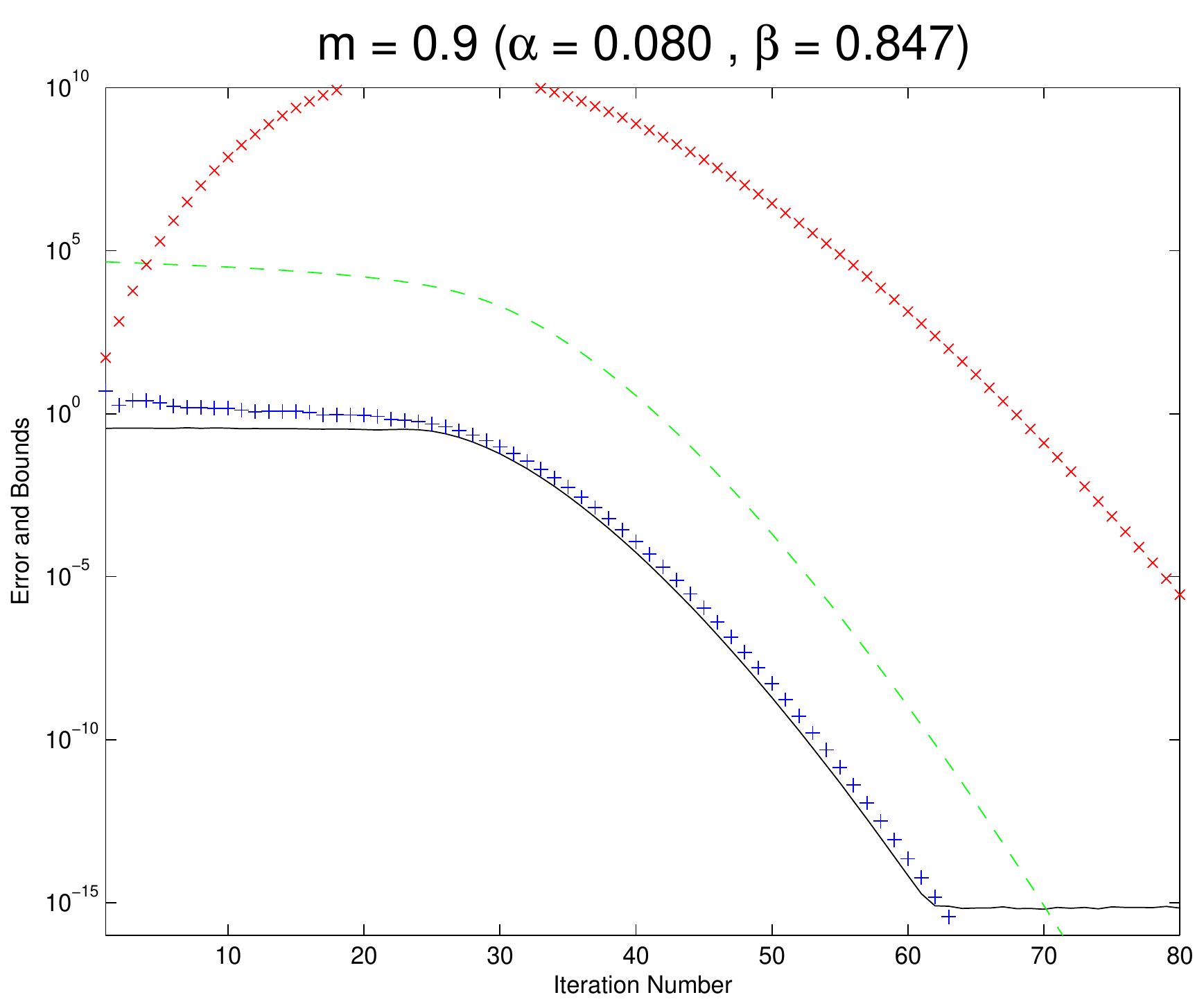}
  \includegraphics[scale=0.4]{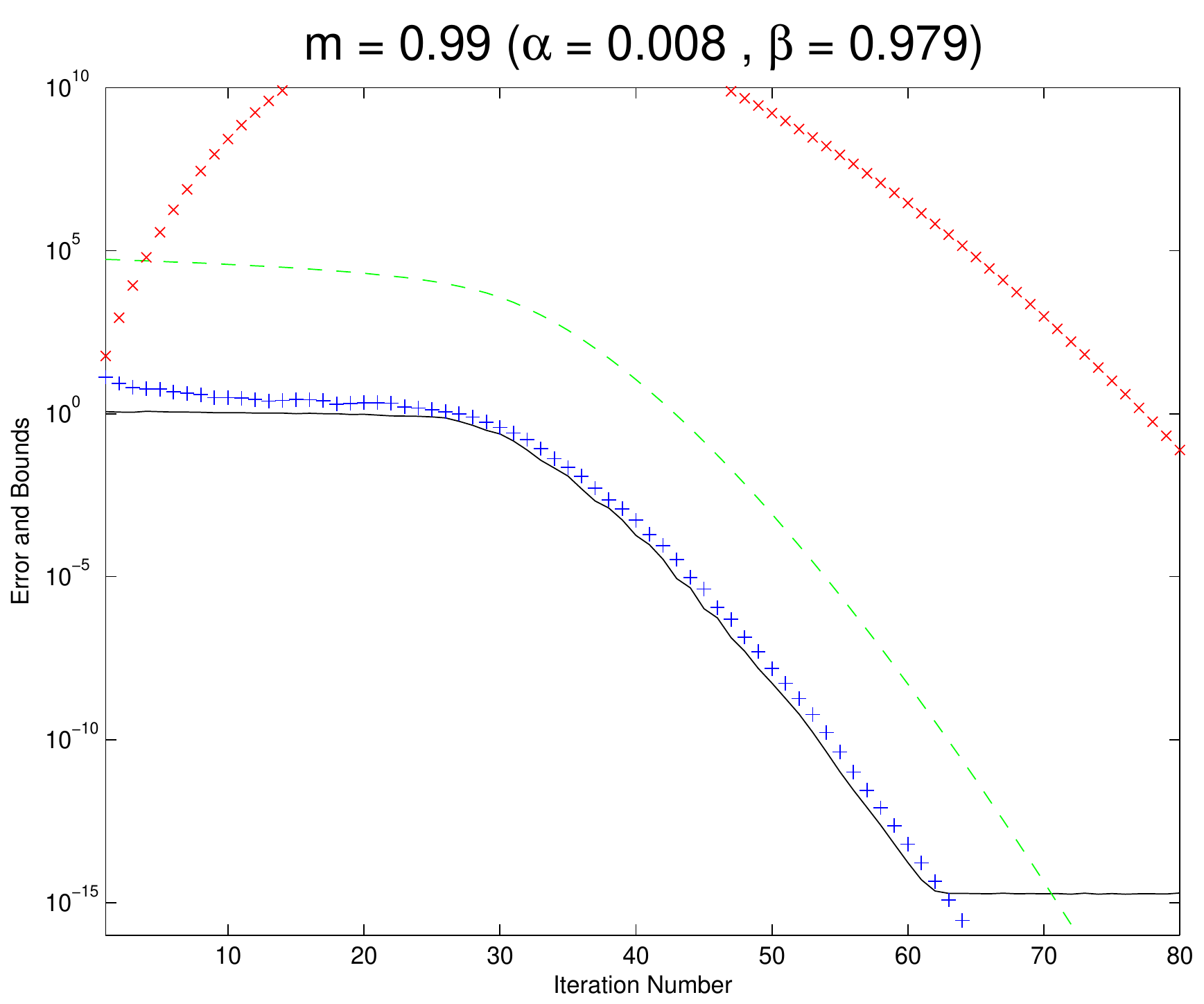}
  \caption
  {
    Example 2.
    $m=0.01,0.1$ (top) and $m=0.9,0.99$ (bottom).
    Error(solid), our {\em a posteriori} bound (+), our {\em a priori} bound (dashed), Saad's bound (x).
  }
  \label{fig:ex2}
\end{figure}

Figure \ref{fig:ex2} shows that the convergence is related to $m$. For smaller $m$ when the eigenvalues lie close to the real axis, the convergence occurs at early iterations and at a faster rate. As $m$ increases to 1, the convergence has an initial stagnation stage before the convergence occurs. Again, this behavior is captured in our new {\em a priori} bound. Our new bound also significantly improves Saad's, which is based on the norm of the matrix only. Our {\em a posteriori} error estimate is sharp for all tests.

We further demonstrate our new bounds for non-positive definite matrices. We construct as in Example 1 a matrix $A$ whose field of values is contained in the square $[\sigma,2+\sigma]\times[-1,1]$ with $\sigma=-1$ and $-10$. We plot in Figure \ref{fig:ex2_shift} the actual error (solid), {\em a posteriori} bound (+), {\em a priori} bound (dashed) and Saad's bound (x). We see that our bounds are still valid when $A$ is not positive definite. They also demonstrate the initial stagnation of convergence. However, the bound becomes more pessimistic for larger $\sigma$.

\begin{figure}[!h]
  \centering
  \includegraphics[scale=0.4]{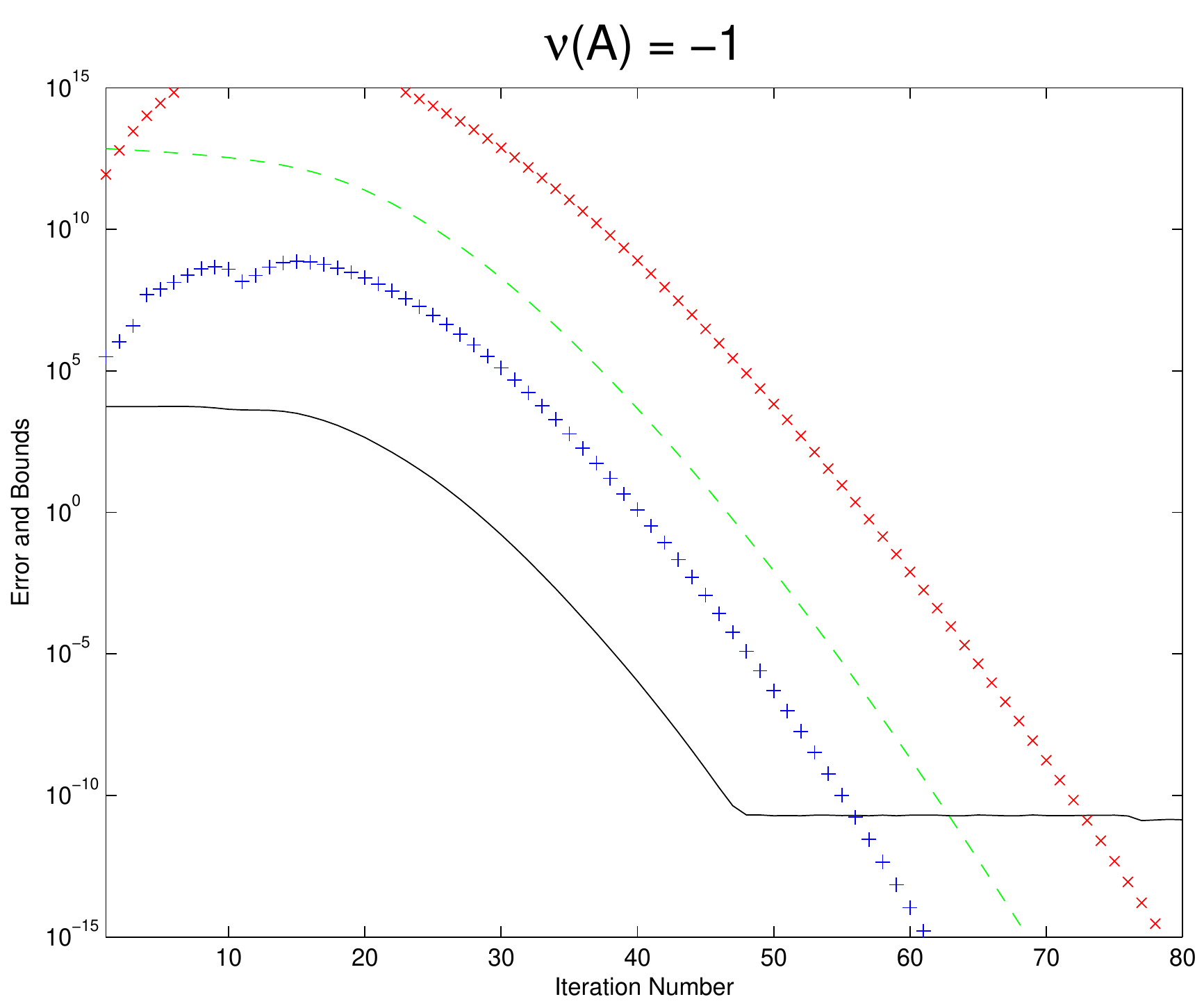}
  \includegraphics[scale=0.4]{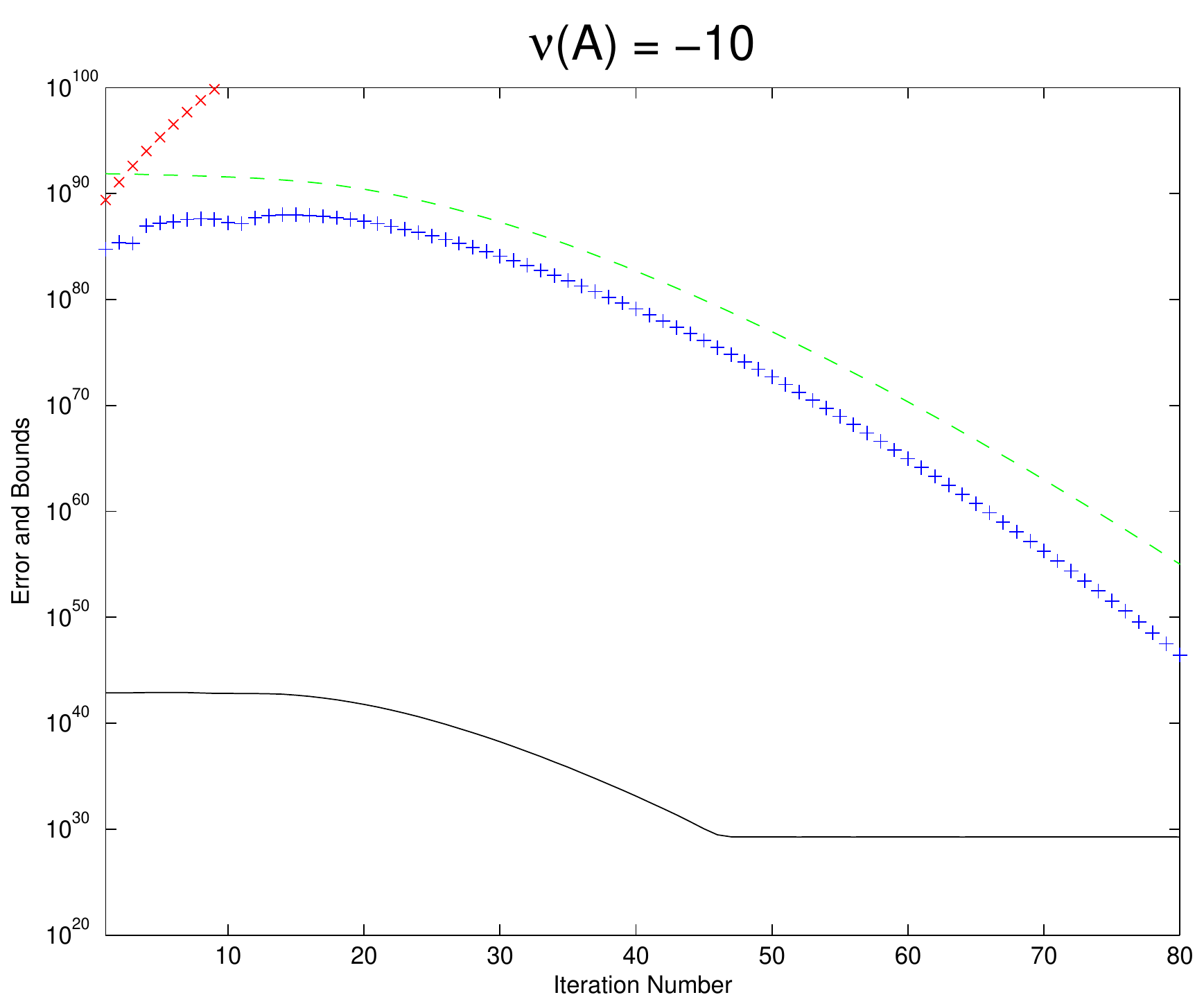}
  \caption
  {
    Example 2.
    Non-positive definite matrix with negative $\nu(A)$.
    Error (solid), our {\em a posteriori} bound (+), our {\em a priori} bound (dashed), and Saad's bound (x).
  }
  \label{fig:ex2_shift}
\end{figure}

In the next example, we consider matrices arising in the convection diffusion equation
\begin{equation}
  \label{eqn:convection_diffusion}
  \frac{\partial}{\partial t}u(x,y)=\triangle u(x,y)-u_x(x,y)-u_y(x,y),\;\;u=0\;\mbox{ in }\;\partial\Omega
\end{equation}
where $(x,y)\in \Omega=[0,1]^2$. The finite-difference discretization in $x,y$ with a uniform mesh leads to an initial value problem \eqref{eqn:ode_nonhomogeneous} and hence the problem of computing $w(\tau)=e^{-\tau A}v$.

\emph{Example 3.}
Let $-A$ be the finite-difference discretization of \eqref{eqn:convection_diffusion} in a $20\times20$ grid in $[0,1]^2$ scaled with $h^2$ so that $||A||_2\approx 8$. Then $A$ is non-Hermitian but positive definite. Let $v$ be a random vector with $||v||_2=1$ and we compute the matrix exponential $w(\tau)=e^{-\tau A}v$. We use various values of $\tau=2,10,20,50$ and apply the Arnoldi method to $A$ and $v$ and the results are presented in Figure \ref{fig:ex3} with $||w(\tau)-w_k(\tau)||$ in the solid line, our {\em a posteriori} error estimate \eqref{eqn:aposteriori_re} in the $+$-line, our {\em a priori} bound \eqref{eqn:apriori_re} in the dashed line and Saad's bound \eqref{eqn:saad_re} in the x-line.

\begin{figure}[!h]
  \centering
  \includegraphics[scale=0.4]{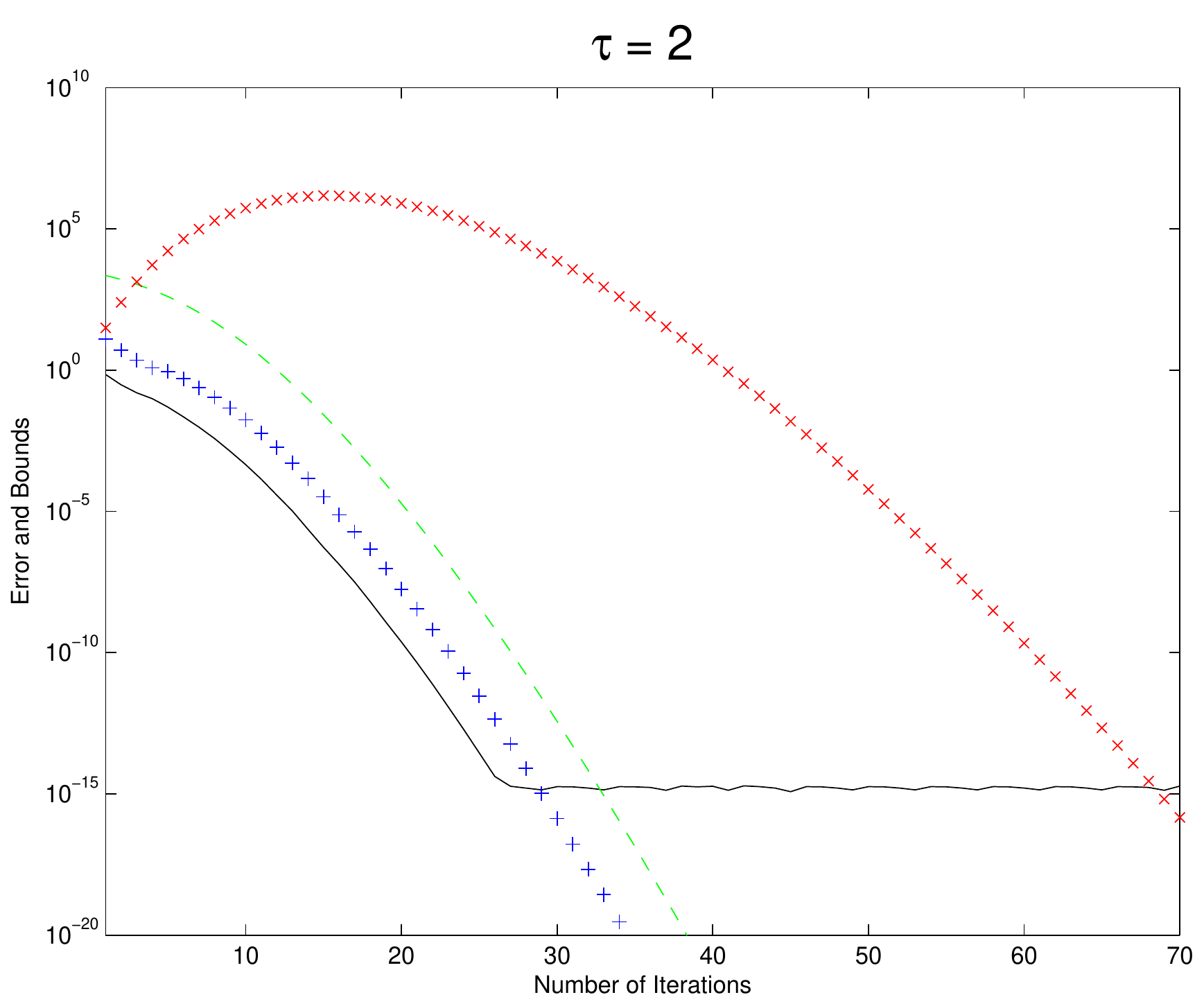}
  \includegraphics[scale=0.4]{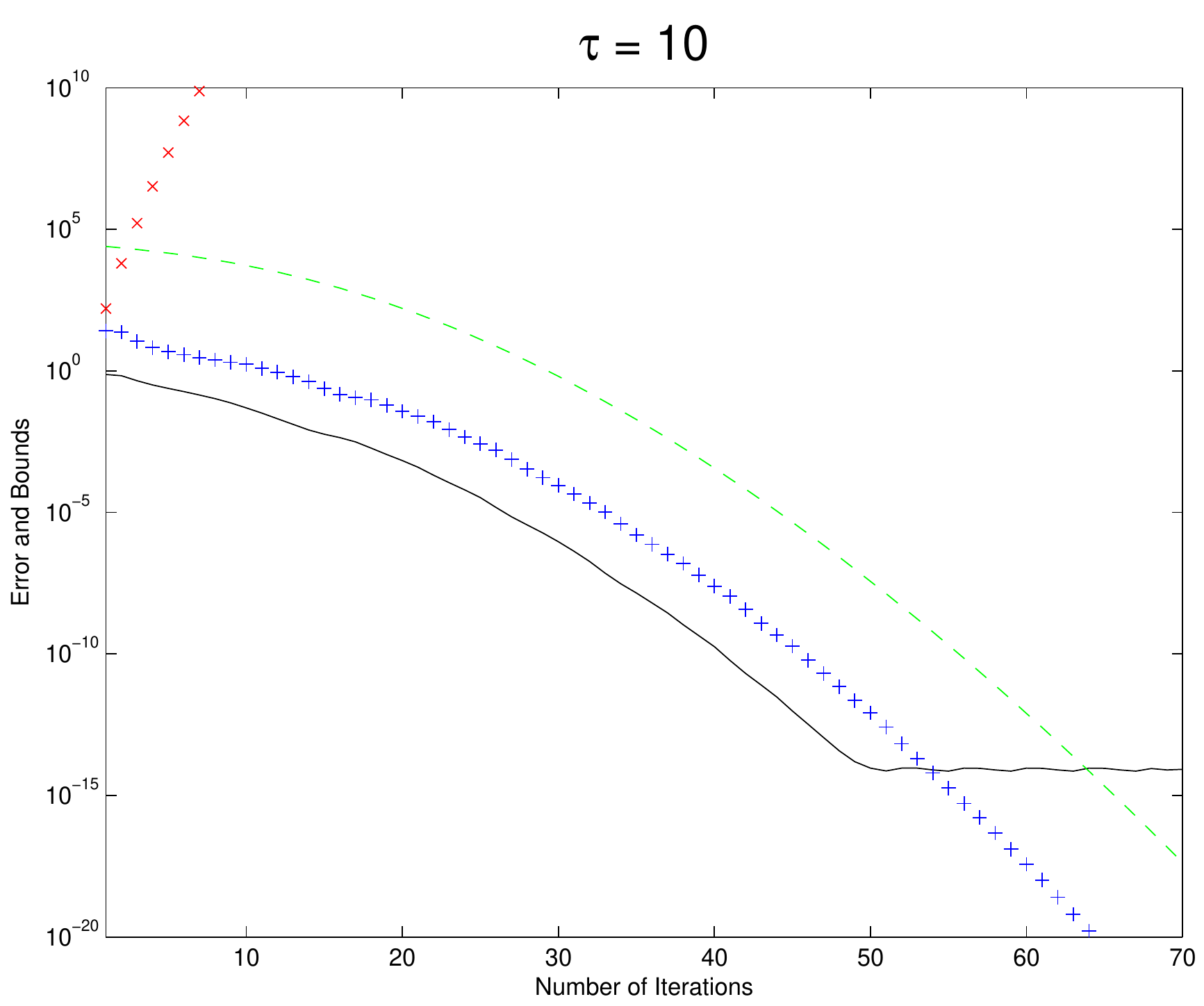}
  \includegraphics[scale=0.4]{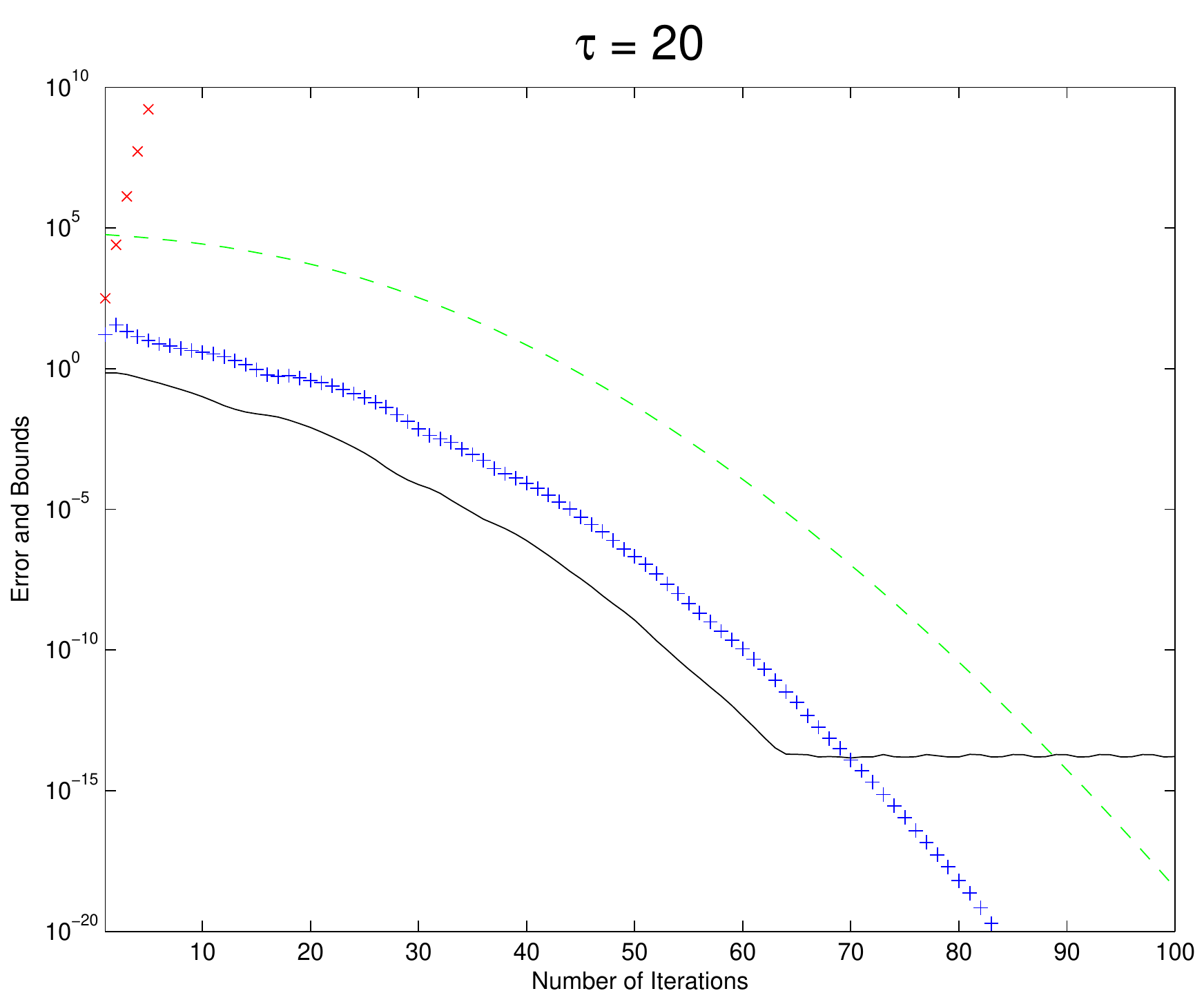}
  \includegraphics[scale=0.4]{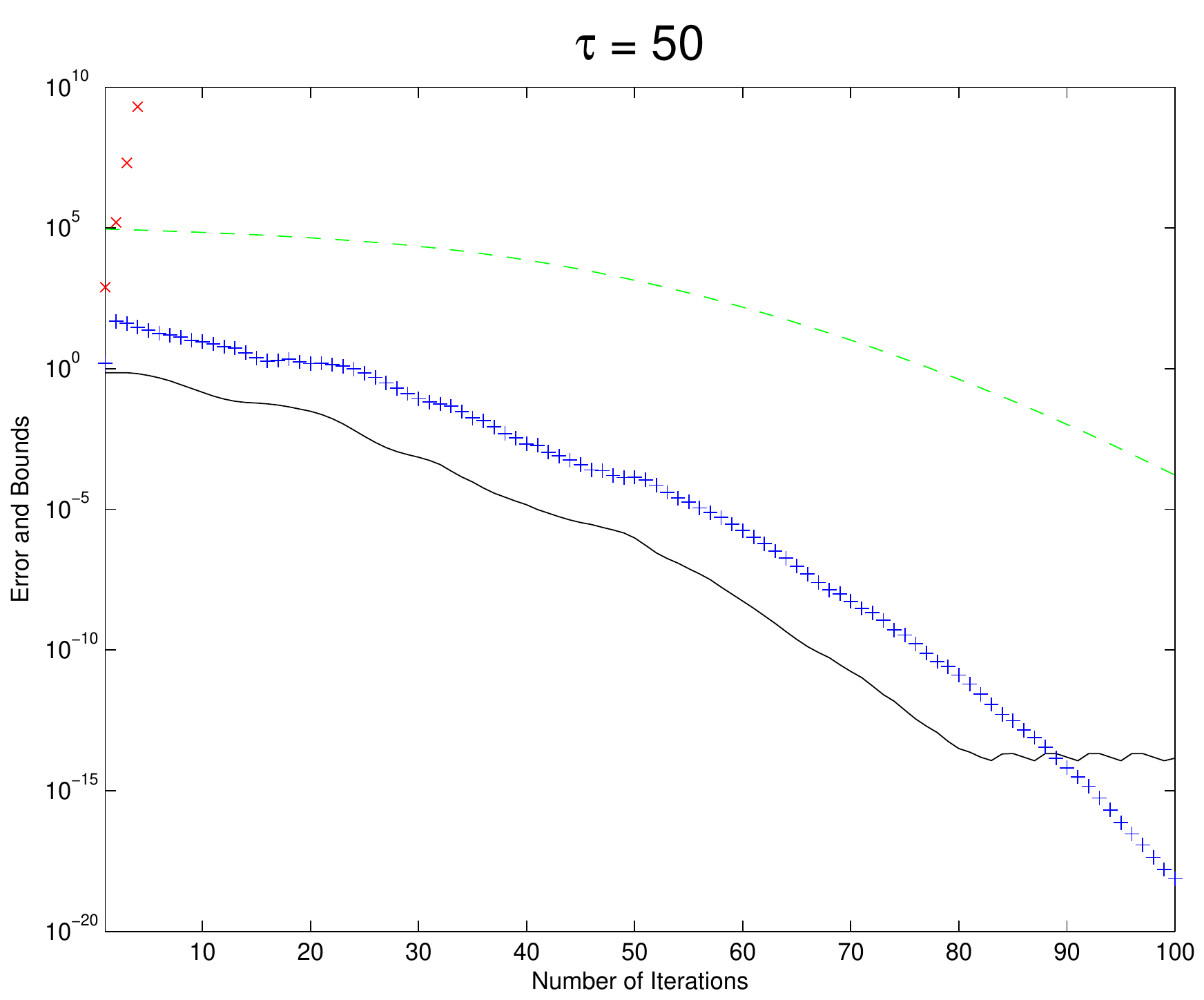}
  \caption
  {
    Example 3.
    $\tau=2,10,20,50$.
    Error(solid), {\em a posteriori} bound (+), {\em a priori} bound (dashed), Saad's bound (x).
  }
  \label{fig:ex3}
\end{figure}

We observe that for $\tau=2$, our {\em a priori} bound is already a significant improvement on the classical bound by Saad. For modestly large values of $\tau$, Saad's bound becomes very pessimistic due to the large norm of $\tau A$, while our {\em a priori} bound still follows the convergence curve of the error. For the case when $\tau=50$ ($\tau||A||_2\approx 400$) or larger, our {\em a priori} bound also becomes very pessimistic. In all the cases, our {\em a posteriori} error estimate remains sharp.

Our final example  concerns skew-Hermitian matrices. 

\emph{Example 4.}
Let $H$ be an $n\times n$ diagonal matrix whose $j$-th diagonal entry is $j/n$. Let $v$ be a random $n\times1$ normalized vector. Then $||H|||_2=1$ and the spectral gap $4\rho=\lambda_{\max}(H)-\lambda_{\min}(H)$ is approximately 1. We apply $k$ iterations of the Lanczos method to compute $w(\tau)=e^{i\tau H}v$. We will test $n=1000$ with $\tau=2,10,20,50$ and the results are presented in Figure \ref{fig:ex4} with $||w(\tau)-w_k(\tau)||$ in the solid line, our {\em a posteriori} error estimate \eqref{eqn:aposteriori_im0} in the $+$-line, our {\em a priori} bound \eqref{eqn:apriori_re} in the dashed line, Hochbruch and Lubich's bound \eqref{eqn:hochbruck_im} in the dotted line, and Saad's bound in the x-line.

\begin{figure}[!h]
  \centering
  \includegraphics[scale=0.4]{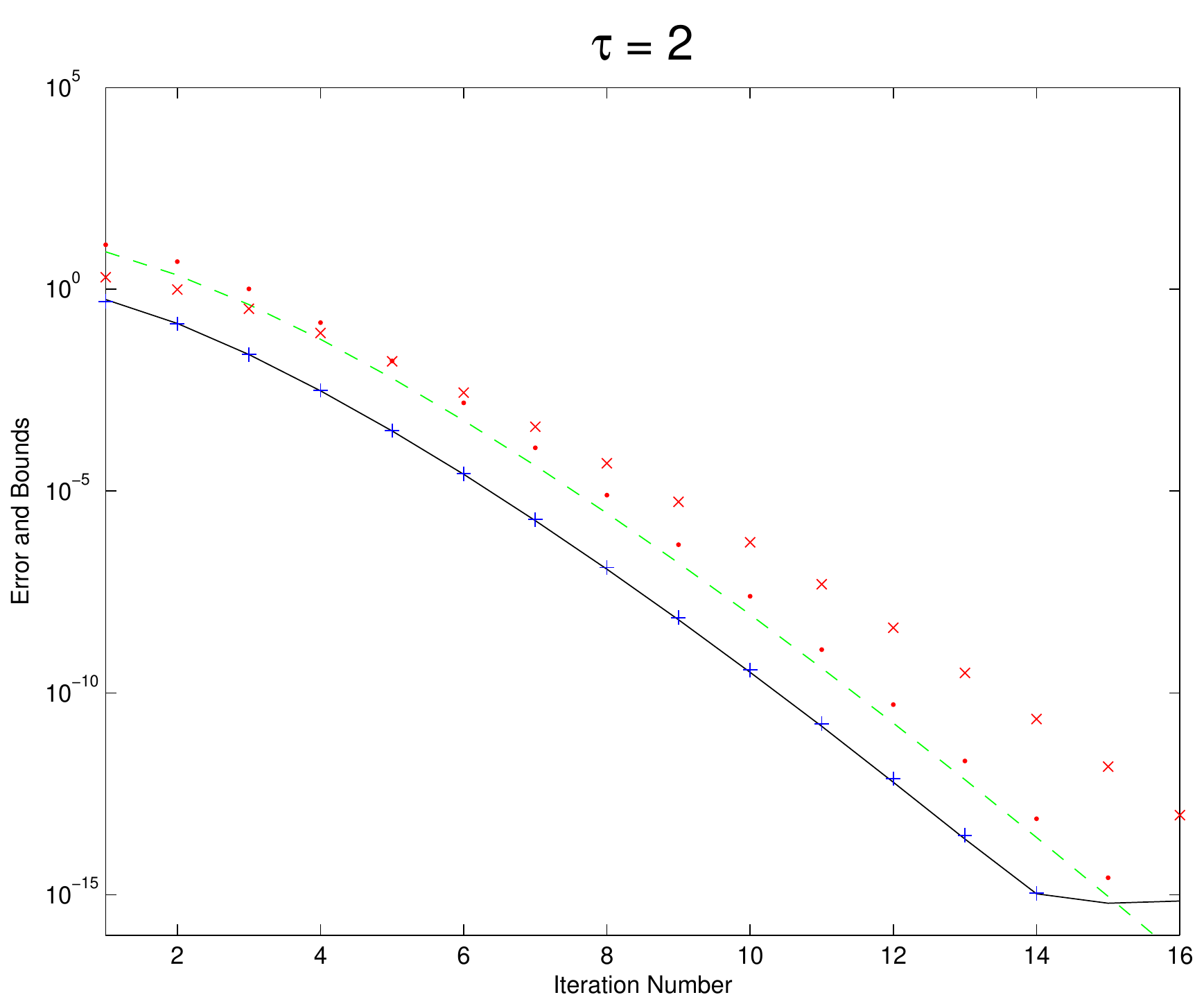}
  \includegraphics[scale=0.4]{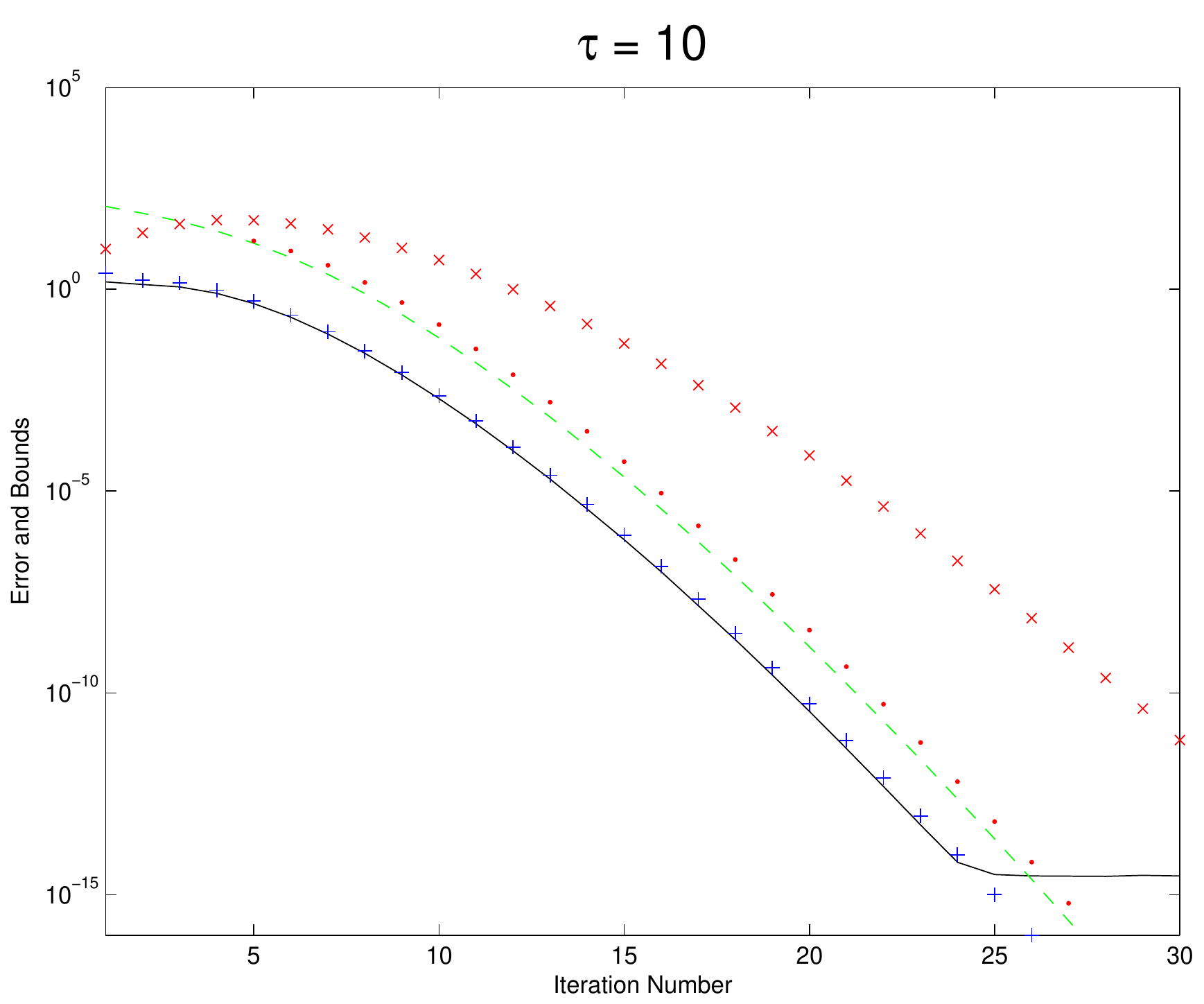}
  \includegraphics[scale=0.4]{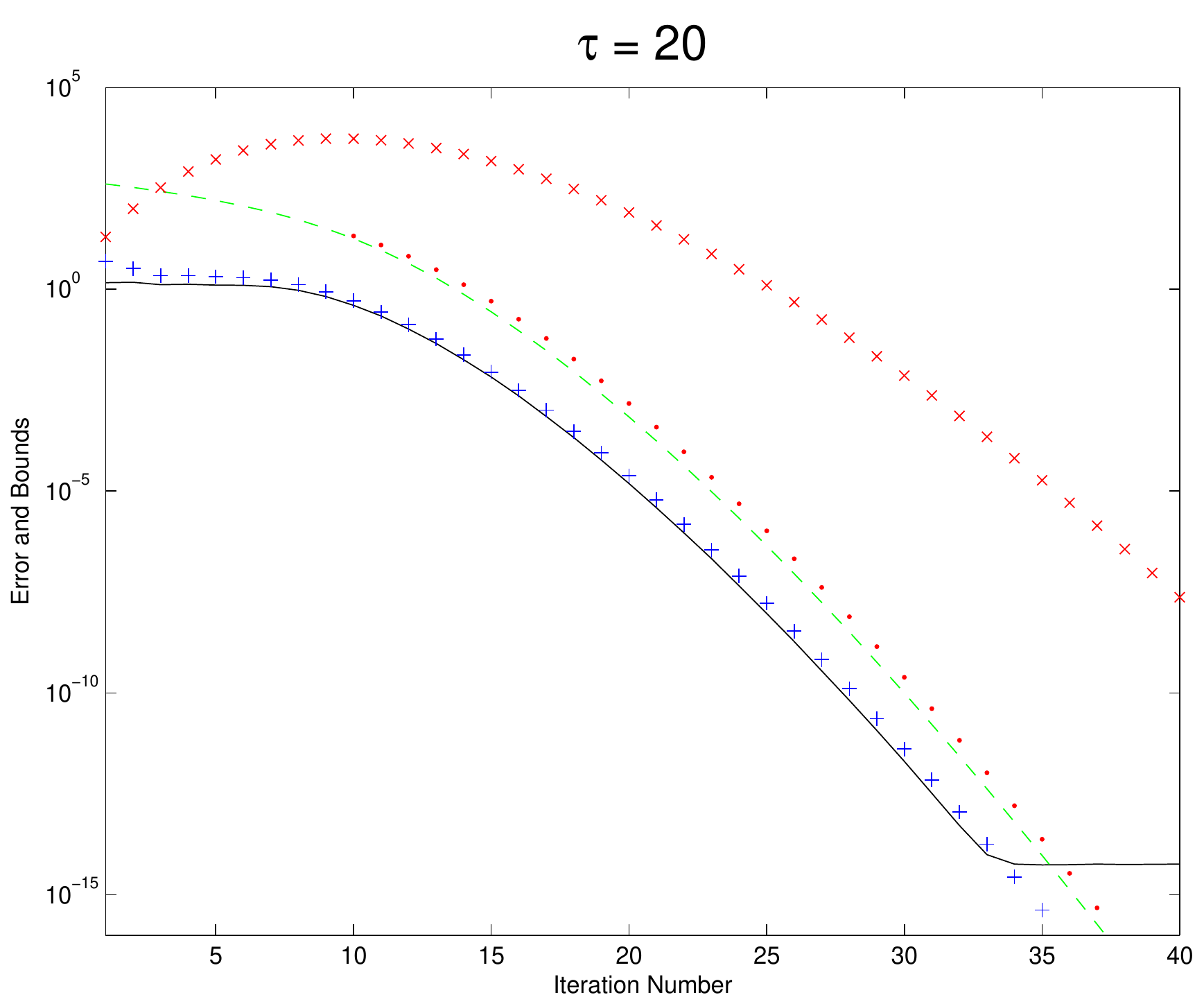}
  \includegraphics[scale=0.4]{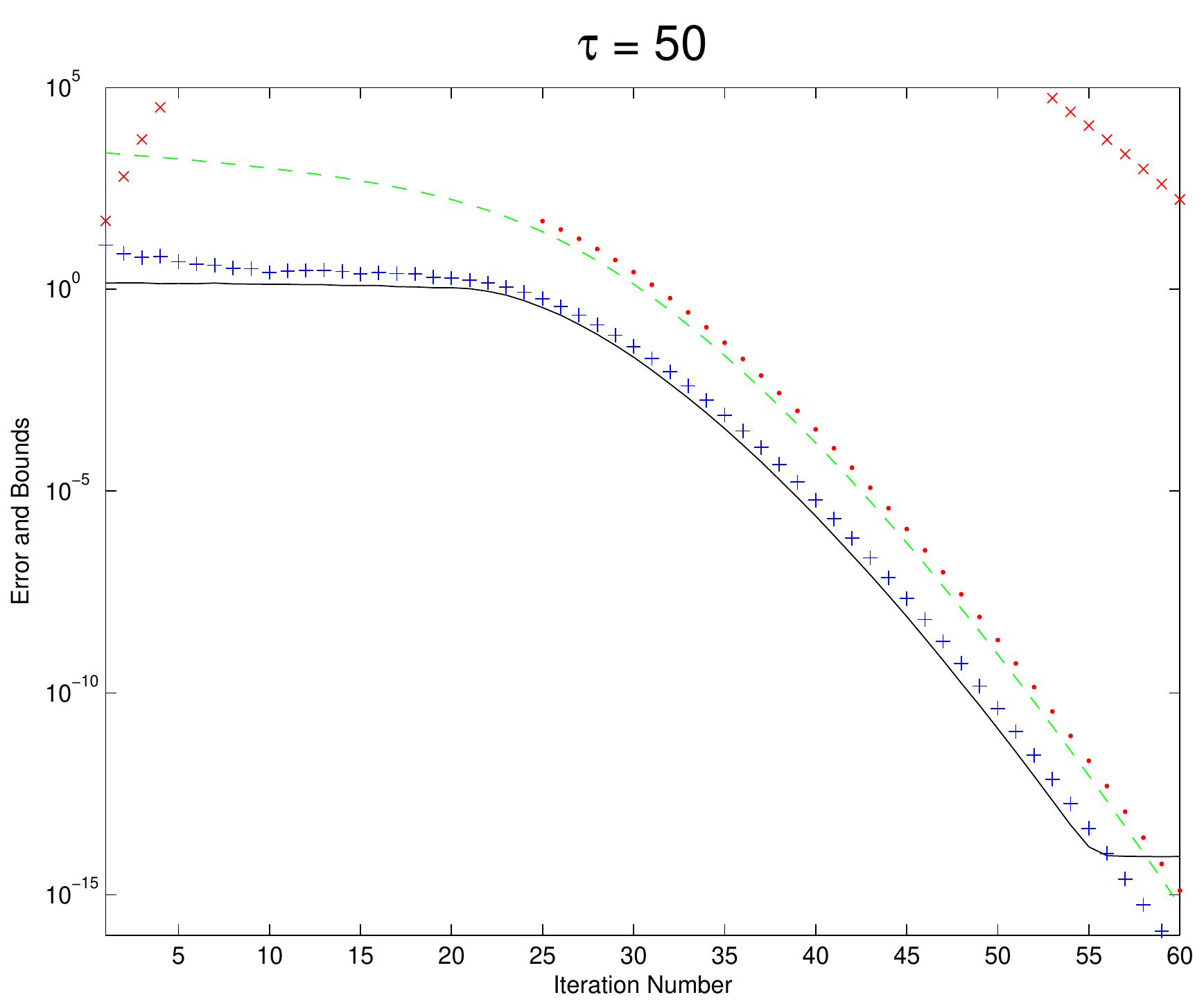}
  \caption
  {
    Example 4.
    $1000\times1000$ diagonal matrix with $a_{jj}=j/1000$.
    $\tau=2,10,20,50$.
    Error (solid), {\em a posteriori} bound (+), {\em a priori} bound (dashed), Hochbruch and Lubich's bound (dotted), and Saad's bound (x).
  }
  \label{fig:ex4}
\end{figure}

We first observe that our bound only  improves Hochbruch and Lubich's bound very slightly. It is significantly better than Saad's bound when $\tau $ is large. In all cases, our and Hochbruch and Lubich's bound follow  the actual error quite closely and our {\em a posteriori} error estimate is sharp. In addition, for larger $\tau$, the error typically stagnates first for some iterations before it starts to converge. An analysis of our bound has shown that the convergence may be expected to start at $k=2\tau \rho $. For $\tau=2,10,20,50$, the corresponding $k$ is 1, 5, 10 and 25, respectively. This basically matches the actual convergence curve in Figure \ref{fig:ex4}, especially when $\tau$ is relatively large and more iterations are needed for the convergence.


\section{Concluding remarks}
\label{sec:concluding_remarks}

For the computation of $e^{-\tau A}v$ with a non-Hermitian matrix $A$ by the Krylov subspace methods, we have presented an {\em a posteriori} error bound that provides a sharp estimate of the error. We have also derive new {\em a priori} error bounds based on the largest and the smallest eigenvalues of the Hermitian and the skew-Hermitian parts of $A$. Using this simple spectral information, our bounds capture convergence characteristics of the Krylov subspace methods. 
They also explain  often observed initial stagnation of the convergence curve.
Numerical comparisons with existing bounds also show that our new bounds may significantly improve the {\em a priori} bound by Hochbruch and Lubich \cite{ye15} that is based on a circular enclosing region of the field of values and the one by Saad \cite{ye24} that is based on the norm. Finally, it agrees with the bound \cite{ye} for the symmetric positive definite case.

The technique developed in this paper provides a new way to analyze convergence of the Krylov subspace method for non-Hermitian matrices through the bounding rectangle for the field of values. It may be extended to other linear algebra problems. For the future works, we plan to study convergence bounds for linear systems based on the Hermitian and the skew-Hermitian parts of $A$, which may also add to the theory of the Krylov subspace method for linear systems.


\bigskip

\noindent {\bf Acknowledgement:} We would like to thank Prof. Michele Benzi for many valuable discussions and in particular for his suggestion to use the technique in \cite{bebo14} that has turned out to be very fruitful.


\end{document}